\documentclass{amsart}

\usepackage{a4}
\usepackage{amsmath} 
\usepackage{amssymb}
\usepackage[all,2cell,emtex]{xy}
\input xypic       
\SilentMatrices

\SelectTips{cm}{10}                               

\theoremstyle{plain} 
\newtheorem*{thm'}{Theorem}                    

\newtheorem{thm}{Theorem}[section]
\newtheorem{prop}[thm]{Proposition}
\newtheorem{lemme}[thm]{Lemma}

\theoremstyle{remark}

\theoremstyle{definition}

\newtheorem{paragr}[thm]{}
\newtheorem{subparagr}{}[thm]

\makeatletter

\newcommand{\ie}{\emph{i.e.}}
\newcommand{\cf}{\emph{cf.}}

\newcommand{\smsp}{\ }
\newcommand{\hhbox}[1]{\qquad\hbox{#1}\qquad}
\newcommand{\ssmash}[1]{\smash{#1}\vrule depth 5pt width 0pt}

\newcommand\DeclareMathOperatorSf[1]{%
  \expandafter\DeclareMathOperator\csname #1\endcsname{\mathsf{#1}}}

\DeclareMathOperatorSf{Hom}
\DeclareMathOperatorSf{Ob}
\newcommand\id[1]{{1^{}_{#1}}}
\newcommand{\Ens}{\mathsf{Set}}
\newcommand{\Top}{\mathsf{Top}}
\newcommand{\Hot}{\mathsf{Hot}}
\newcommand{\Cat}{\mathsf{Cat}}
\newcommand{\op}{\mathrm{op}}
\newcommand{\Ar}[1]{\mathcal{A}r(#1)}

\newcommand{\pref}[1]{\widehat{#1}}

\newcommand{\To}[1]{{\hskip -2.5pt\xymatrixcolsep{1pc}\xymatrix{\ar[r]^{#1}&}\hskip -2.5pt}}
\newcommand{\tosim}{\To{\sim}}

\newcommand{\Toto}[2]{{\hskip -2.5pt\xymatrixcolsep{#1pc}\xymatrix{\ar[r]^{#2}&}\hskip -2.5pt} }
\renewcommand{\to}{\Toto{.7}{}}
\newcommand{\toto}{\Toto{1.3}{}}

\newcommand{\mathUniv}[1]{\mathbf{#1}}
\newcommand{\mathUnivGr}[1]{\mathchoice{\mbox{\boldmath $#1$}}{\mbox{\boldmath $#1$}}{\mbox{\boldmath $\scriptstyle #1$}}{\mbox{\boldmath $\scriptscriptstyle #1$}}}

\newcommand{\Glob}{\mathbb{G}}

\newcommand{\Disk}{D}
\newcommand{\cosource}{\sigma}
\newcommand{\cotarget}{\tau}

\newcommand{\Dn}[1]{\mathUniv{\Disk}_{#1}}

\newcommand{\Tht}[2][\@empty]{%
  \ifx\@empty#1
  \mathUnivGr{\cotarget}_{#2}^{}%
  \else
    \mathUnivGr{\cotarget}_{#1}^{#2}%
  \fi
}

\newcommand{\Ths}[2][\@empty]{%
  \ifx\@empty#1
    \mathUnivGr{\cosource}_{#2}^{}%
  \else
    \mathUnivGr{\cosource}_{#1}^{#2}%
  \fi
}

\newcommand{\C}{C}

\newcommand{\ImD}[1]{\Disk_{#1}}

\newcommand{\Imt}[2][\@empty]{%
  \ifx\@empty#1
    \cotarget_{#2}^{}%
  \else
    \cotarget_{#1}^{#2}%
  \fi
}

\newcommand{\Ims}[2][\@empty]{%
  \ifx\@empty#1
    \cosource_{#2}^{}%
  \else
    \cosource_{#1}^{#2}%
  \fi
}

\newcommand{\Thetazero}{\mathUnivGr{\Theta}_0}

\newcommand{\Gr}{$\mathsf{Gr}$}

\newcommand{\Coh}{\mathUniv{C}}
\newcommand{\Cohcan}{\Coh_{\mathrm{can}}}
\newcommand{\CohBL}{\Coh_{\mathrm{BL}}}
\newcommand{\Cohred}{\Coh_{\mathrm{red}}}

\newcommand{\Adm}{E}

\newcommand{\Grp}{G}
\newcommand{\GrpFl}[1]{\Grp_{#1}}
\newcommand{\GrpFlQ}[1]{\overline{\Grp}_{#1}}
\newcommand{\GrpCat}[1]{\mathsf{Gr}_\infty^{#1}}

\newcommand{\source}{s}
\newcommand{\target}{t}

\newcommand{\Grpt}[2][\@empty]{%
  \ifx\@empty#1
  \target_{#2}^{}%
  \else
    \target_{#1}^{#2}%
  \fi
}

\newcommand{\Grps}[2][\@empty]{%
  \ifx\@empty#1
    \source_{#2}^{}%
  \else
    \source_{#1}^{#2}%
  \fi
}

\newcommand{\can}[1]{\mathUniv{can}_{#1}^{}}
\newcommand{\comult}{\mathUnivGr{\nabla}}

\newcommand{\comultbin}[2][\@empty]{%
  \ifx\@empty#1
    \comult_{}^{#2}%
  \else
    \comult_{#1}^{#2}%
  \fi
}

\newcommand{\comultbinprim}[2][\@empty]{%
  \ifx\@empty#1
    \comult'{}_{}^{#2}%
  \else
    \comult'{}_{#1}^{#2}%
  \fi
}

\def\comp{\mathop{\ast}}
\def\compprim{\mathop{\ast'}}

\newcommand{\multbin}[4][\@empty]{%
  \ifx\@empty#1
    #3 \comp\limits^{#2} #4%
  \else
    #3 \comp\limits_{#1}^{#2} #4%
  \fi
}

\newcommand{\multbinprim}[4][\@empty]{%
  \ifx\@empty#1
    #3 \compprim\limits^{#2} #4%
  \else
    #3 \compprim\limits_{#1}^{#2} #4%
  \fi
}

\newcommand{\coassoc}{\mathUnivGr{\alpha}}

\newcommand{\coass}[2][\@empty]{%
  \ifx\@empty#1
    \coassoc_{}^{#2}%
  \else
    \coassoc_{#1}^{#2}%
  \fi
}

\newcommand{\assoc}{a}

\newcommand{\ass}[2][\@empty]{%
  \ifx\@empty#1
    \assoc_{}^{#2}%
  \else
    \assoc_{#1}^{#2}%
  \fi
}

\newcommand{\counit}[1]{\mathUnivGr{\kappa}^{}_{#1}}
\newcommand{\unit}[1]{k^{}_{#1}}
\newcommand{\varunit}[1]{\mathsf{id}_{#1}}
\newcommand{\counitlcontr}[1]{\mathUnivGr{\lambda}^{}_{#1}}
\newcommand{\counitrcontr}[1]{\mathUnivGr{\rho}^{}_{#1}}
\newcommand{\unitlcontr}[1]{l^{}_{#1}}
\newcommand{\unitrcontr}[1]{r^{}_{#1}}

\newcommand{\coinverse}{\mathUnivGr{\omega}}

\newcommand{\coinv}[2][\@empty]{%
  \ifx\@empty#1
    \coinverse_{}^{#2}%
  \else
    \coinverse_{#1}^{#2}%
  \fi
}

\newcommand{\inverse}{w}

\newcommand{\inv}[2][\@empty]{%
  \ifx\@empty#1
    \inverse_{}^{#2}%
  \else
    \inverse_{#1}^{#2}%
  \fi
}
 \newcommand{\varinv}[1]{#1^{-1}}

\newcommand{\coinvlcontr}[1]{\mathUnivGr{\gamma}^{}_{#1}}
\newcommand{\coinvrcontr}[1]{\mathUnivGr{\delta}^{}_{#1}}
\newcommand{\invlcontr}[1]{g^{}_{#1}}
\newcommand{\invrcontr}[1]{d^{}_{#1}}

\newcommand{\hmtp}[1]{\sim_{#1}}
\newcommand{\Pizero}[1]{\Pi_0(#1)}

\newcommand{\GrptQ}[1]{\overline{\target}_{#1}}
\newcommand{\GrpsQ}[1]{\overline{\source}_{#1}}

\newcommand{\GrpMor}{f}

\newcommand{\TopD}[1]{\Disk_{#1}}

\newcommand{\Topt}[2][\@empty]{%
  \ifx\@empty#1
    \cotarget_{#2}^{}%
  \else
    \cotarget_{#1}^{#2}%
  \fi
}

\newcommand{\Tops}[2][\@empty]{%
  \ifx\@empty#1
    \cosource_{#2}^{}%
  \else
    \cosource_{#1}^{#2}%
  \fi
}

\newcommand{\Sphere}[1]{S_{#1}}
\newcommand{\bordD}[1]{\partial\TopD{#1}}

\newcommand{\infGrpf}{\Pi_\infty}
\newcommand{\EspClass}{B}

\newcommand{\CC}{\mathcal{C}}
\newcommand{\Classarr}{\mathcal{F}}
\newcommand{\LClassarr}{\mathcal{A}}
\newcommand{\RClassarr}{\mathcal{B}}
\newcommand{\Setarr}{I}
\newcommand{\wellord}{J}
\newcommand{\filt}{J}
\newcommand{\llp}[1]{l(#1)}
\newcommand{\rlp}[1]{r(#1)}
\newcommand{\cof}[1]{\mathsf{cof}(#1)}
\newcommand{\Cell}{\mathsf{cell}}

\newcommand{\cell}[2][\@empty]{%
  \ifx\@empty#1
    \Cell(#2)%
  \else
    \Cell_{#1}(#2)%
  \fi
}

\newcommand{\mc}[2]{\mathchoice {\raise -1.8pt\vbox{\hbox{$#1\backslash$}}#2} {\raise -1.8pt\vbox{\hbox{$#1\backslash$}}#2} {\raise -1.8pt\vbox{\hbox{$\scriptstyle#1\backslash$}}#2} {\raise -1.8pt\vbox{\hbox{$\scriptscriptstyle#1\backslash$}}#2} }

\newcommand{\CatExtGl}{\mc{\Thetazero}{\mathsf{Gl\,\hbox{\textrm{-}}\,Ext}}}
\newcommand{\GlobUniv}[1]{\overline{#1}}

\newcommand{\Tree}{\mathUniv{X}}
\newcommand{\SourceI}[2]{\Thetazero[#1,#2]}
\newcommand{\ButI}[2]{\Thetazero[#1,#2]'}
\newcommand{\SetArr}{I}

\newcommand{\Thbord}{\mathUnivGr{\partial}}
\newcommand{\Thbordind}[1]{\overline{#1\kern -3pt}\kern 3pt}
\newcommand{\Thbords}[1]{\mathUnivGr{\epsilon}^{}_{#1}}
\newcommand{\Thbordt}[1]{\mathUnivGr{\eta}^{}_{#1}}

\newcommand{\Ctg}{$\mathsf{Cat}$}
\newcommand{\CohCat}{\mathUniv{C}}
\newcommand{\Catg}{F}

\newcommand{\imds}[1]{\mathsf{imds}(#1)}
\newcommand{\Square}{\mathcal{D}}
\newcommand{\minwellord}{0}
\newcommand{\wo}{j}
\newcommand{\smwo}{{j'}}
\newcommand{\bgwo}{j}
\newcommand{\Succ}[1]{#1+1}

\makeatother

\renewcommand{\hookrightarrow}{{\hskip -1.5pt\raise 1.5pt\vbox{\xymatrixcolsep{.9pc}\xymatrix{\ar@{^{(}->}[r]&}}\hskip -3.5pt}}
\renewcommand{\longmapsto}{{\hskip -2.5pt\xymatrixcolsep{1.3pc}\xymatrix{\ar@{|->}[r]&}\hskip -2.5pt}}
\renewcommand{\mapsto}{{\hskip -2.5pt\xymatrixcolsep{.8pc}\xymatrix{\ar@{|->}[r]&}\hskip -2.5pt}}

\title[Grothendieck $\infty$-groupoids, and a definition of lax $\infty$-categories]{{\Large Grothendieck $\infty$-groupoids,} \break 
\vskip -6pt
and still another definition of $\infty$-categories}
\author{Georges Maltsiniotis}

\address{Institut de Math\'ematiques de Jussieu\\
Universit\'e Paris 7 Denis Diderot\\
\hfill\break\indent Case Postale~7012\\
B\^atiment Chevaleret\\
F-75205, Paris Cedex 13\\
FRANCE}
\email{maltsin@math.jussieu.fr}
\urladdr{http://www.math.jussieu.fr/\raise -3.3pt\vbox{\hbox{$\widetilde{ \ }\,$}}maltsin/}

\begin{document}

\begin{abstract}
The aim of this paper is to present a simplified version of the notion of $\infty$-groupoid developed by Grothendieck in ``Pursuing Stacks'' and to introduce a definition of $\infty$-categories inspired by Grothendieck's approach.
\end{abstract}

\maketitle

\section*{Introduction}

The precise definition of Grothendieck $\infty$\nobreakdash-groupoids~\cite[sections 1-13]{PS} has been presented in~\cite{Ma1}. In this paper, we give a slightly simplified version of this notion, and a variant leading to a definition of (weak) $\infty$\nobreakdash-categories~\cite{Ma2}, very close to Batanin's operadic definition~\cite{Bat}. The precise relationship between these two notions is investigated by Ara in~\cite{Ara}.
\medbreak

The basic intuition leading to the definition of a $\infty$-groupoid is presented as follows by Grothendieck (for a $\infty$\nobreakdash-groupoid $F$, with set of $i$\nobreakdash-cells $F_i$): ``\emph{Intuitively, it means that whenever we have \emph{two} ways of associating to a finite family $(u_i)_{i \in I}$ of objects of an $\infty$\nobreakdash-groupoid, $u_i \in F_{n(i)}$, subjected to a ``standard'' set of relations on the $u_i$'s, an element of some $F_n$, in terms of the $\infty$\nobreakdash-groupoid structure only, then we have automatically a ``homotopy'' between these built-in in the very structure of the $\infty$\nobreakdash-groupoid, provided it makes at all sense to ask for one}\,\dots''~\cite[section 9]{PS}. This leads him to the notion of \emph{coherator}, category $\Coh$ endowed with a ``universal $\infty$\nobreakdash-cogroupoid'', a $\infty$\nobreakdash-groupoid being a presheaf on $\Coh$ satisfying some left exactness conditions, improperly called Segal conditions in the literature. In particular, Grothendieck $\infty$\nobreakdash-groupoids define an algebraic structure species, and the category of $\infty$\nobreakdash-groupoids is locally presentable.
\medbreak

The notion of a $\infty$\nobreakdash-groupoid depends on the choice of a coherator. Two different coherators give rise in general to non-equivalent categories of $\infty$\nobreakdash-groupoids. Nevertheless, the two notions of $\infty$\nobreakdash-groupoid are expected to be equivalent in some subtler way. Grothendieck illustrates this fact as follows: ``\emph{Roughly saying, two different 
mathematicians, working independently on the conceptual problem I had in 
mind, assuming they both wind up with some explicit definition, will 
almost certainly get non-equivalent definitions -- namely with non-equivalent 
categories of \emph{(}set-valued, say\emph{)} $\infty$\nobreakdash-groupoids! And, 
secondly and as importantly, that \emph{this ambiguity however is an 
irrelevant one}. To make this point a little clearer, I could say that a 
third mathematician, informed of the work of both, will readily think out 
a functor or rather a pair of functors, associating to any structure of 
Mr.~X one of Mr.~Y and conversely, in such a way that by composition of 
the two, we will associate to a $X$\nobreakdash-structure \emph{(}$T$ say\emph{)} another $T'$, 
which will not be isomorphic to $T$ of course, but endowed with a 
canonical $\infty$\nobreakdash-equivalence \emph{(}in the sense of Mr.~X\emph{)} $T 
\underset{\infty}{\simeq} T'$, and the same on the Mr.~Y side. Most 
probably, a fourth mathematician, faced with the same situation as the 
third, will get his own pair of functors to reconcile Mr.~X and Mr.~Y, 
which very probably won't be equivalent \emph{(}I mean isomorphic\emph{)} to the previous 
one. Here however a fifth mathematician, informed about this new 
perplexity, will probably show that the two $Y$\nobreakdash-structures $U$ and $U'$, 
associated by his two colleagues to an $X$\nobreakdash-structure $T$, while not 
isomorphic, also admit however a canonical $\infty$\nobreakdash-equivalence between $U$ 
and $U'$ \emph{(}in the sense of the $Y$\nobreakdash-theory\emph{)}. I could go on with a sixth 
mathematician, confronted with the same perplexity as the previous one, who 
winds up with another $\infty$\nobreakdash-equivalence between $U$ and $U'$ \emph{(}without 
being informed of the work of the fifth\emph{),} and a seventh reconciling them by 
discovering an $\infty$\nobreakdash-equivalence between these equivalences. The story 
of course is infinite, I better stop with seven mathematicians,}\,\dots''~\cite[section 9]{PS}.
\medbreak

One of the reasons of Grothendieck's interest in $\infty$\nobreakdash-groupoids is that (weak) $\infty$\nobreakdash-groupoids are conjectured to modelize homotopy types (it's well known that \emph{strict} $\infty$\nobreakdash-groupoids don't): ``\emph{Among the things to be checked is of course that when we localize the 
category of $\infty$\nobreakdash-groupoids with respect to morphisms which are ``weak 
equivalences'' in a rather obvious sense \emph{(N.B. --} the definition of the 
$\pi_i$'s of an $\infty$\nobreakdash-groupoid is practically trivial!\emph{),} we get a category 
equivalent to the usual homotopy category $\Hot$.}''~\cite[section 12]{PS}.
This conjecture is still \emph{not} proven, for any definition of $\infty$-groupoid giving rise to an \emph{algebraic} structure species, although some progress has been done in this direction by Cisinski~\cite{Ci}. It becomes tautological if we \emph{define} $\infty$\nobreakdash-groupoids as being Kan complexes or topological spaces! But the categories of such are \emph{not} locally presentable. In a letter to Tim Porter, Grothendieck clearly explains that this was not the kind of definition he was looking for: ``\emph{my main point is that your suggestion that Kan complexes are ``the ultimate in lax $\infty$\nobreakdash-groupoids'' does not in any way meet with what I am really looking for, and this for a variety of reasons,\,\dots}''\cite{Corr}.
\medbreak

One of the peculiarities of Grothendieck's definition of $\infty$\nobreakdash-groupoids is that this notion is \emph{not} a particular case of a concept of lax $\infty$\nobreakdash-category. Nevertheless, it was realized in \cite{Ma2} that a slight modification of the notion of coherator gives rise to such a concept. This new formalization of lax $\infty$\nobreakdash-categories is very close, although not exactly equivalent, to the notion introduced by Batanin~\cite{Bat}. From a technical point of view, the basic difference is that the first is based on universal algebra, whereas the second on a generalization of the notion of operads, the globular operads. The precise relationship between the two concepts is studied in~\cite{Ara}.
\medbreak

In the first section, Grothendieck's definition of $\infty$-groupoids is presented (in a two-page slightly simplified form), introducing the notion of ``coherator for a theory of $\infty$\nobreakdash-groupoids''. Some examples of such coherators are given. The aim of the very long (too long?) last subsection~\ref{structmor} is to convince the reader of the pertinence of Grothendieck's concept, and define some structural maps in $\infty$\nobreakdash-groupoids, useful in the next section.
\medbreak

In section 2, homotopy groups and weak equivalences between $\infty$\nobreakdash-groupoids are introduced. The pair of adjoint functors ``classifying space'' of a $\infty$\nobreakdash-groupoid and ``fundamental $\infty$\nobreakdash-groupoid'' of a topological space are defined. Grothendieck's conjecture is presented.
\medbreak

In section 3, the only really original part of this paper, an interpretation of the notion of coherator for a theory of $\infty$\nobreakdash-groupoids, in terms of lifting properties and weak factorization systems, is given.
\medbreak

In the last section, the definition of $\infty$-categories of~\cite{Ma2} is presented, introducing the notion of ``coherator for a theory of $\infty$\nobreakdash-categories''. The reader mainly interested by this definition can read directly this section after~\ref{cohgr} and skip everything in between.
\medbreak

In appendix A, a technical result used in section 3 is proved.
\goodbreak

\section{Grothendieck $\infty$-groupoids}

\begin{paragr} \textbf{The category of globes.} \label{catglob}
The \emph{globular category} or \emph{category of globes} is the category $\Glob$ generated by the graph
\[
\xymatrix{
\Dn{0} \ar@<.8ex>[r]^-{\Ths{1}} \ar@<-.8ex>[r]_-{\Tht{1}} &
\Dn{1} \ar@<.8ex>[r]^-{\Ths{2}} \ar@<-.8ex>[r]_-{\Tht{2}} &
\ \ \cdots\ \ \ar@<.8ex>[r]^-{\Ths{i-1}} \ar@<-.8ex>[r]_-{\Tht{i-1}} &
\Dn{i-1} \ar@<.8ex>[r]^-{\Ths{i}} \ar@<-.8ex>[r]_-{\Tht{i}} &
\Dn{i} \ar@<.8ex>[r]^-{\Ths{i+1}} \ar@<-.8ex>[r]_-{\Tht{i+1}} &
\ \ \cdots\smsp\smsp,
}
\]
under the \emph{coglobular} relations
\[
\Ths{i+1}\Ths{i} = \Tht{i+1}\Ths{i}\quad\text{and}\quad\Ths{i+1}\Tht{i} =
\Tht{i+1}\Tht{i}\smsp, \qquad i \geqslant 1\smsp.
\]
For every $i,\,j$, such that $0\leqslant i\leqslant j$, define
\[
\Ths[i]{j} = \Ths{j}\cdots\Ths{i+2}\Ths{i+1}\quad\text{and}\quad
  \Tht[i]{j} = \Tht{j}\cdots\Tht{i+2}\Tht{i+1}\smsp,
\]
and observe that
\[
\Hom_\Glob(\Dn{i},\Dn{j}) = 
  \begin{cases}
  \{\Ths[i]{j}, \Tht[i]{j}\}\smsp, & \text{if $i < j$\smsp,}\\
  \{\id{\Dn{i}}\}\smsp, & \text{if $i = j$}\smsp,\\
  \varnothing\smsp, & \text{else.}
  \end{cases}
\]
\end{paragr}

\begin{paragr} \textbf{Globular sums.} \label{sommesglob}
Let $\C$ be a category, $\Glob\to\C$ a functor, and let $\ImD{i}$, $\Ims{i}$, $\Imt{i}$, $\Ims[i]{j}$, $\Imt[i]{j}$ be the image in $\C$ of $\Dn{i}$, $\Ths{i}$, $\Tht{i}$, $\Ths[i]{j}$, $\Tht[i]{j}$ respectively.
A \emph{standard iterated amalgamated sum}, or more simply a \emph{globular sum}, of \emph{length} $m$, in $\C$ is an iterated amalgamated sum of the form
\[
(\ImD{i_1},\Ims[i'_1]{i_1})\amalg^{}_{\ImD{i'_1}}(\Imt[i'_1]{i_2},\ImD{i_2},\Ims[i'_2]{i_2})\amalg^{}_{\ImD{i'_2}}\cdots\,\,\amalg^{}_{\ImD{i'_{m-1}}}(\Imt[i'_{m-1}]{i_m},\ImD{i_m})\smsp,
\]
colimit of the diagram
\[
\hskip 20pt
\xymatrixrowsep{.2pc}
\xymatrixcolsep{1pc}
\xymatrix{
\ImD{i_1} &  & \ImD{i_2} &  & \ImD{i_3} &        & \ImD{i_{m - 1}} & & \ImD{i_{m}} \\
  &  &   &  &   & \cdots &     & & \\
  & \ImD{i'_1}
  \ar[uul]^{\Ims[i'_1]{i_1}}
  \ar[uur]_{\negthickspace \Imt[i'_1]{i_2}}  &   &
  \ImD{i'_2}' 
  \ar[uul]^{\Ims[i'_2]{i_2}\negthickspace}
  \ar[uur]_{\Imt[i'_2]{i_3}}  &  &  & &
\ImD{i'_{m-1}} \ar[uul]^{\Ims[i'_{m-1}]{i_{m-1}}}  \ar[uur]_{\Imt[i'_{m-1}]{i_m}}
& & \hskip -20pt\text{,}
}
\]
where $m\geqslant1$, and for every $k$, $1\leqslant k<m$, $i'_k$ is strictly smaller then $i_k$ and $i_{k+1}$. Such a globular sum is completely determined by the \emph{table of dimensions}
\[
\left(
\begin{matrix}
i_1 && i_2 && \cdots && i_m \cr
& i'_1 && i'_2 & \cdots & i'_{m-1}
\end{matrix}
\right)\smsp,
\]
and will be simply denoted
\[
\ImD{i_1}\amalg^{}_{\ImD{i'_1}}\ImD{i_2}\amalg^{}_{\ImD{i'_2}}\cdots\,\,\amalg^{}_{\ImD{i'_{m-1}}}\ImD{i_m}\smsp.
\]
\end{paragr}

\begin{paragr} \textbf{Globular extensions.} \label{extglob}
A category $\C$, endowed with a functor $\Glob\to\C$, is called a \emph{globular extension} if globular sums exist in $\C$. For example, \emph{any} functor from $\Glob$ to a cocomplete category defines a globular extension. A morphism from a globular extension to another is a functor under $\Glob$, commuting with globular sums. There exists a universal globular extension $\Glob\to\Thetazero$, called a \emph{globular completion of $\Glob$}, satisfying the following universal property: for every globular extension $\Glob\to\C$, there exists a morphism of globular extensions $\Thetazero\to\C$, unique up to unique natural isomorphism (inducing the identity on objects coming from $\Glob$). This universal globular extension, defined up to equivalence of categories, can be constructed, for example, by taking the closure by globular sums of $\Glob$ embedded by the Yoneda functor in the category $\pref{\Glob}$ of \emph{globular sets} (or $\infty$\nobreakdash-\emph{graphs}), \ie~presheaves on $\Glob$. The objects of $\Thetazero$ are \emph{rigid}, \ie~have no non-trivial automorphisms, and there is a combinatorial description of the category $\Thetazero$ in terms of planar trees~\cite{Ara,Ber,Ma2}, leading to a \emph{skeletal} incarnation of $\Thetazero$ (such that isomorphic objects are equal), the objects being in bijection with tables of dimensions. In the sequel we choose, once and for all, such a skeletal model of $\Thetazero$.
\end{paragr}

\begin{paragr} \textbf{\boldmath Coherators for a theory of $\infty$-groupoids.} \label{cohgr}
Let $\C$ be a category, $\Glob\to\C$ a functor, and let $\ImD{i}$, $\Ims{i}$, $\Imt{i}$ be the image in $\C$ of $\Dn{i}$, $\Ths{i}$, $\Tht{i}$ respectively. A \emph{pair of parallel arrows} in $\C$ is a pair $(f,g)$ of arrows $f,g:\ImD{i}\to X$ in $\C$ such that either $i=0$, or $i>0$ and $f\Ims{i}=g\Ims{i}$, $f\Imt{i}=g\Imt{i}$. A \emph{lifting} of such a pair $(f,g)$ is an arrow $h:\ImD{i+1}\to X$ such that $f=h\Ims{i+1}$, $g=h\Imt{i+1}$. A \emph{coherator for a theory of $\infty$\nobreakdash-groupoids}, or more simply a \emph{\Gr-coherator}, is a globular extension $\Glob\to\Coh$ satisfying the following two conditions:
\smallbreak

\emph{a}) Every pair of parallel arrows in $\Coh$ has a lifting in $\Coh$.
\smallskip

\emph{b}) There exists a ``tower'' of globular extensions (called \emph{tower of definition} of the \Gr\nobreakdash-coherator $\Coh$) with colimit $\Coh$
\[
\xymatrixcolsep{1.5pc}
\xymatrix{
\Glob \ar[r]
&\Coh_0\ar[r]
&\Coh_1\ar[r]
&\cdots\ar[r]
&\Coh_n\ar[r]
&\Coh_{n+1}\ar[r]
&\cdots\ar[r]
&\Coh\simeq\varinjlim\Coh_n
}\smsp,
\]
where for every $n\geqslant0$, $\Coh_n$ is a small category, $\Coh_n\to\Coh_{n+1}$ a morphism of globular extensions, and 
satisfying the following properties:
\begin{itemize}
\item[$b_0$\hskip 1pt)] $\Glob\to\Coh_0$ is a globular completion;
\item[$b_n$)] for every $n\geqslant0$, there exists a family of pairs of parallel arrows in $\Coh_n$ such that $\Coh_{n+1}$ is the universal globular extension obtained from $\Coh_n$ by formally adding a lifting for every pair in this family.
\end{itemize}
Condition $(b_0)$ implies that the category $\Coh_0$ is equivalent to $\Thetazero$. We will usually assume that $\Coh_0$ is \emph{equal} to $\Thetazero$ (and that the functor $\Glob\to\Coh_0$ is the canonical inclusion \hbox{$\Glob\to\Thetazero$}). Condition $(b_n)$ means, more precisely, that there exists a family $(\mathUniv{f}_i,\mathUniv{g}_i)_{i\in I_n}$  of pairs of parallel arrows in $\Coh_n$, and for every $i\in I_n$, a lifting $\mathUniv{h}_i$ in $\Coh_{n+1}$ of the image of the pair $(\mathUniv{f}_i,\mathUniv{g}_i)$ in $\Coh_{n+1}$, satisfying the following universal property. For every globular extension $\Glob\to\C$ and every morphism of globular extensions $\Coh_n\to\C$, if for every $i\in I_n$ a lifting $h_i$ of the image of the pair $(\mathUniv{f}_i,\mathUniv{g}_i)$ in $\C$ is given, then there exists a unique morphism of globular extensions $F:\Coh_{n+1}\to\C$ such that for every $i$ in $I_n$, $F(\mathUniv{h}_i)=h_i$ and such that the triangle
\[
\xymatrixrowsep{1.7pc}
\xymatrix{
\Coh_n\ar[r]\ar[rd]
&\Coh_{n+1}\ar[d]^{F}
\\
&\C
}
\]
is commutative. It can be easily seen that the functors $\Coh_n\to\Coh_{n+1}$ induce bijections on the sets of objects, so that we can suppose that all categories $\Coh_n$ and $\Coh$ have same objects, indexed by tables of dimensions (or planar trees). Furthermore, it can be proved (\cf~\ref{fondinfgrpd}) that for any \Gr\nobreakdash-coherator $\Coh$, the induced functor $\Thetazero=\Coh_0\to\Coh$ is faithful. In particular, the category $\Glob$ will be usually identified to a (non-full) subcategory of $\Coh$. It is conjectured that all functors $\Coh_n\to\Coh$ are faithful.
\end{paragr}

\begin{paragr} \textbf{\boldmath Grothendieck $\infty$-groupoids.} \label{definfgr}
Let $\Glob\to\Coh$ be a \Gr\nobreakdash-coherator. A \emph{$\infty$\nobreakdash-groupoid of type $\Coh$}, or more simply a \emph{$\infty$\nobreakdash-$\Coh$\nobreakdash-groupoid}, is a presheaf $\Grp:\Coh^\op\to\Ens$ on $\Coh$ such that the functor $\Grp^\op:\Coh\to\Ens^\op$ preserves globular sums. In other terms, for any globular sum in $\Coh$, the canonical map
\[
\Grp(\Dn{i_1}\amalg^{}_{\Dn{i'_1}}\cdots\,\,\amalg^{}_{\Dn{i'_{m-1}}}\Dn{i_m})\toto
\Grp(\Dn{i_1})\times^{}_{\Grp(\Dn{i'_1})}\cdots\,\,\times^{}_{\Grp(\Dn{i'_{m-1}})}\Grp(\Dn{i_m})
\smsp
\]
is a bijection, the right hand side being the \emph{standard iterated fiber product} or \emph{globular product}, limit of the diagram
\[
\xymatrixrowsep{.2pc}
\xymatrixcolsep{1pc}
\xymatrix{
\Grp(\Dn{i_1})\ar[rdd]_(.37){\Grp(\Ths[i'_1]{i_1})}
&&\Grp(\Dn{i_2})\ar[ldd]^(.37){\!\Grp(\Tht[i'_1]{i_2})}\ar[rdd]^(.65){\!\!\!\!\Grp(\Ths[i'_2]{i_2})}
&&&&\Grp(\Dn{i_m})\ar[ldd]^(.37){\!\Grp(\Tht[i'_{m-1}]{i_m})}
\\
&&&&\cdots
\\
&\Grp(\Dn{i'_1})
&&\Grp(\Dn{i'_2})
&&\Grp(\Dn{i'_{m-1}})
&\hskip 40pt.
}
\] 
The \emph{category of $\infty$\nobreakdash-$\Coh$\nobreakdash-groupoids} is the full subcategory of $\pref{\Coh}$, category of presheaves on $\Coh$, whose objects are $\infty$\nobreakdash-$\Coh$\nobreakdash-groupoids.
\end{paragr}

\begin{paragr} \textbf{Examples of \Gr-coherators.} \label{excohgr}
There is a general method for constructing inductively \Gr\nobreakdash-coherators. Take $\Coh_0=\Thetazero$. Suppose that $\Coh_n$ is defined and choose a set $\Adm_n$ of pairs of parallel arrows in $\Coh_n$.
Define $\Coh_{n+1}$ as the universal globular extension obtained by formally adding a lifting for each pair in $\Adm_n$ (an easy categorical argument shows that such a universal globular extension exists, is unique up to unique isomorphism, and that the functor $\Coh_n\to\Coh_{n+1}$ induces a bijection on the sets of objects~\cite[section~2.6]{Ara}). Let $\Coh$ be the colimit $\Coh=\varinjlim\Coh_n$. For an arbitrary choice of the sets $\Adm_n$, $\Coh$ need not be a \Gr\nobreakdash-coherator, as there is no reason for condition (\emph{a}) in~\ref{cohgr} to be satisfied. A sufficient (but not necessary) condition for $\Coh$ to be a \Gr\nobreakdash-coherator is that every pair of parallel arrows in $\Coh$ is the image of a pair in~$\Adm_n$, for some $n\geqslant0$. Three important examples can be constructed (among many others) by this method.
\smallbreak

1) \textbf{The canonical \Gr-coherator $\Coh=\Cohcan$\,.} This example is obtained by taking $\Adm_n$ to be the set of \emph{all} pairs of parallel arrows in $\Coh_n$. 
\smallbreak

2) \textbf{The Batanin-Leinster \Gr-coherator $\Coh=\CohBL$\,.} It is obtained by defining $\Adm_n$ to be the set of pairs of parallel arrows in $\Coh_n$ that are not the image of a pair in $\Adm_{n'}$, for some $n'<n$~\cite[4.1.4]{Ara}.

3) \textbf{The canonical reduced \Gr-coherator $\Coh=\Cohred$\,.} It is constructed by taking $\Adm_n$ to be the set of  pairs of parallel arrows in $\Coh_n$ that do not have already a lifting in $\Coh_n$.
\smallbreak

It is easily seen that examples 1 and 2 satisfy the sufficient condition stated above. The example 3 does \emph{not} satisfy this condition; nevertheless, it is clear that $\Cohred$ \emph{is} a \Gr\nobreakdash-coherator. It is possible to put even more restrictive conditions on the sets $\Adm_n$ and still obtain a \Gr\nobreakdash-coherator. It seems that it is not possible to find a minimal way for choosing the sets $\Adm_n$.
\end{paragr}

\begin{paragr} \textbf{Some structural maps.} \label{structmor}
Fix a \Gr-coherator $\Coh$ and a $\infty$\nobreakdash-$\Coh$\nobreakdash-groupoid\break 
$\Grp:\Coh^\op\to\Ens$. The restriction of $\Grp$ to the subcategory $\Glob$ of $\Coh$ defines a $\infty$\nobreakdash-graph, called the \emph{underlying $\infty$\nobreakdash-graph}:
\[
\xymatrixcolsep{2.2pc}
\xymatrix{
\GrpFl{0}
&\GrpFl{1} \ar@<.6ex>[l]^(.45){\Grpt{1}}\ar@<-.9ex>[l]_(.45){\Grps{1}}  
&\quad\cdots\quad\ar@<.6ex>[l]^(.55){\Grpt{2}}\ar@<-.9ex>[l]_(.55){\Grps{2}}
&\GrpFl{i-1}\ar@<.6ex>[l]^(.43){\Grpt{i-1}}\ar@<-.9ex>[l]_(.43){\Grps{i-1}}
&\GrpFl{i}\ar@<.6ex>[l]^(.41){\Grpt{i}}\ar@<-.9ex>[l]_(.41){\Grps{i}}
&\quad\cdots\ar@<.6ex>[l]^{\Grpt{i+1}}\ar@<-.9ex>[l]_{\Grps{i+1}}
}\smsp,
\]
where $\GrpFl{i}=\Grp(\Dn{i})$, $\Grps{i}=\Grp(\Ths{i})$ and $\Grpt{i}=\Grp(\Tht{i})$. The elements of $\GrpFl{i}$ are the \emph{$i$\nobreakdash-cells} of $\Grp$, and $\Grps{i}$, $\Grpt{i}$ are the \emph{source} and \emph{target} maps respectively, satisfying the \emph{globular relations}:
\[
\Grps{i}\Grps{i+1}=\Grps{i}\Grpt{i+1}\quad,\qquad\Grpt{i}\Grps{i+1}=\Grpt{i}\Grpt{i+1}\smsp,\qquad i\geqslant1\smsp.
\] 
Operations and coherence arrows in the $\infty$\nobreakdash-$\Coh$\nobreakdash-groupoid $\Grp$ are defined using the existence of lifting arrows for pairs of parallel arrows in the \Gr\nobreakdash-coherator $\Coh$. In what follows, we give some examples of such structural maps of $\Grp$ (for more details see~\cite[section~4.2]{Ara}). When there is no ambiguity, let's denote by
\[
\xymatrixcolsep{3pc}
\xymatrix{
\Dn{i_k}\ar[r]^-{\can{k}}
&\raise -12pt\hbox{$\Dn{i_1}\amalg^{}_{\Dn{i'_1}}\cdots\,\,\amalg^{}_{\Dn{i'_{m-1}}}\Dn{i_m}$}
}
\]
the canonical map of the $k$-th summand into a globular sum.
\smallbreak

\begin{subparagr} \label{bincompl1}
\textbf{Level 1 binary composition.} For every $i\geqslant1$, the two composite arrows
\[
\xymatrixrowsep{.7pc}
\xymatrix{
\Dn{i-1}\ar[r]^-{\Ths{i}}
&\Dn{i}\ar[r]^-{\can{2}}
&\Dn{i}\amalg_{\Dn{i-1}}\Dn{i}
\\
\Dn{i-1}\ar[r]^-{\Tht{i}}
&\Dn{i}\ar[r]^-{\can{1}}
&\Dn{i}\amalg_{\Dn{i-1}}\Dn{i}
},
\]
form a pair of parallel arrows in $\Coh$, therefore there is a lifting $\comultbin{i}:\,=\comultbin[1]{i}$
such that
\[
\xymatrixrowsep{2.7pc}
\xymatrixcolsep{4.7pc}
\xymatrix{
\Dn{i}\ar@<.5ex>[rd]^{\comultbin{i}=\comultbin[1]{i}}
&
& \ar@{}@<-5ex>[r]^{\textstyle \comultbin{i}\Ths{i}=\can{2}\Ths{i}\smsp,} 
  \ar@{}@<-10ex>[r]^{\textstyle \comultbin{i}\Tht{i}=\can{1}\Tht{i}\smsp.}
&
\\
\,\Dn{i-1} \ar@<1ex>[u]^{\Ths{i}} \ar@<-1ex>[u]_{\Tht{i}}
           \ar@<1ex>[r]^-{\can{2}\Ths{i}} \ar@<-1ex>[r]_-{\can{1}\Tht{i}}
&\Dn{i}\amalg_{\Dn{i-1}}\Dn{i} 
}
\]
We deduce a map 
\[
\xymatrixrowsep{2.7pc}
\xymatrixcolsep{2.7pc}
\xymatrix{
\GrpFl{i}\times_{\GrpFl{i-1}}\GrpFl{i}\simeq\Grp(\Dn{i}\amalg_{\Dn{i-1}}\Dn{i}) \ar[r]^-{\Grp(\comultbin{i})}
&\Grp(\Dn{i})=\GrpFl{i}\smsp,
}
\]
associating to each pair of $i$-cells $(x,y)$ such that $\Grps{i}(x)=\Grpt{i}(y)$ a $i$\nobreakdash-cell 
\[
\multbin{}{x}{y}:\,=\multbin{i}{x}{y}:\,=\multbin[1]{i}{x}{y}:\,=\Grp(\comultbin{i})(x,y)
\]
such that
\[
\Grps{i}(\multbin{}{x}{y})=\Grps{i}(y)\hhbox{and}\Grpt{i}(\multbin{}{x}{y})=\Grpt{i}(x)\smsp.
\]
This defines a ``\emph{vertical}'' or ``\emph{level $1$}'' \emph{composition} of $i$-cells. The lifting $\comultbin{i}$ need not be unique, but if $\comultbinprim{i}$ is such another lifting (defining another vertical composition $\multbinprim{}{}{}$ on $i$\nobreakdash-cells), then $(\comultbin{i},\comultbinprim{i})$ is a pair of parallel arrows in $\Coh$, and there exists a lifting $\mathUnivGr{\Gamma}:\Dn{i+1}\to\Dn{i}\amalg_{\Dn{i-1}}\Dn{i}$ such that $\mathUnivGr{\Gamma}\,\Ths{i+1}=\comultbin{i}$ and $\mathUnivGr{\Gamma}\,\Tht{i+1}=\comultbinprim{i}$, hence a map $c:\GrpFl{i}\times_{\GrpFl{i-1}}\GrpFl{i}\to\GrpFl{i+1}$, associating to each pair $(x,y)$ of ``composable'' $i$\nobreakdash-cells a ``homotopy'' $(i+i)$\nobreakdash-cell $c(x,y)$ of source $\multbin{}{x}{y}$ and target $\multbinprim{}{x}{y}$.
\end{subparagr}

\begin{subparagr} \label{bincompl2}
\textbf{Level 2 binary composition.} Given for every $i\geqslant1$ a lifting $\comultbin{i}=\comultbin[1]{i}$ as above, observe that, for $i\geqslant2$, the two composite arrows
\[
\xymatrixcolsep{4.2pc}
\xymatrixrowsep{.7pc}
\xymatrix{
\Dn{i-1}\ar[r]^-{\comultbin[1]{i-1}}
&\Dn{i-1}\amalg_{\Dn{i-2}}\Dn{i-1}\ar[r]^-{\Ths{i}\amalg_{\Dn{i-2}}\Ths{i}}
&\Dn{i}\amalg_{\Dn{i-2}}\Dn{i}
\\
\Dn{i-1}\ar[r]^-{\comultbin[1]{i-1}}
&\Dn{i-1}\amalg_{\Dn{i-2}}\Dn{i-1}\ar[r]^-{\Tht{i}\amalg_{\Dn{i-2}}\Tht{i}}
&\Dn{i}\amalg_{\Dn{i-2}}\Dn{i}
},
\]
form a pair of parallel arrows in $\Coh$. Therefore there is a lifting $\comultbin[2]{i}$
such that
\[
\xymatrixrowsep{3.5pc}
\xymatrixcolsep{3.2pc}
\xymatrix{
\Dn{i}\ar@<.5ex>[rrd]^{\comultbin[2]{i}}
&
&
& \ar@{}@<-6ex>[r]^{\textstyle\qquad\comultbin[2]{i}\Ths{i}=(\Ths{i}\amalg_{\Dn{i-2}}\Ths{i})\comultbin[1]{i-1}\smsp,} 
  \ar@{}@<-11ex>[r]^{\textstyle\qquad\comultbin[2]{i}\Tht{i}=(\Tht{i}\amalg_{\Dn{i-2}}\Tht{i})\comultbin[1]{i-1}\smsp.}
&
\\
\,\Dn{i-1} \ar@<1ex>[u]^{\Ths{i}} \ar@<-1ex>[u]_{\Tht{i}}
           \ar@<1ex>[rr]^-{(\Ths{i}\amalg_{\Dn{i-2}}\Ths{i})\comultbin[1]{i-1}\ } 
           \ar@<-1ex>[rr]_-{(\Tht{i}\amalg_{\Dn{i-2}}\Tht{i})\comultbin[1]{i-1}\ }
&&\Dn{i}\amalg_{\Dn{i-2}}\Dn{i} 
}
\]
We deduce a map 
\[
\xymatrixrowsep{2.7pc}
\xymatrixcolsep{2.7pc}
\xymatrix{
\GrpFl{i}\times_{\GrpFl{i-2}}\GrpFl{i}\simeq\Grp(\Dn{i}\amalg_{\Dn{i-2}}\Dn{i}) \ar[r]^-{\Grp(\comultbin[2]{i})}
&\Grp(\Dn{i})=\GrpFl{i}
},
\]
associating to each pair of $i$-cells $(x,y)$ such that the iterated source of $x$ in $\GrpFl{i-2}$ is equal to the iterated target of $y$, a $i$\nobreakdash-cell (level 2 composition of $x$ and $y$)
\[
\multbin[2]{i}{x}{y}:\,=\Grp(\comultbin[2]{i})(x,y)
\]
such that
\[
\Grps{i}(\multbin[2]{i}{x}{y})=\multbin[1]{i-1}{\Grps{i}(x)}{\Grps{i}(y)}\hhbox{and}\Grpt{i}(\multbin[2]{i}{x}{y})=\multbin[1]{i-1}{\Grpt{i}(x)}{\Grpt{i}(y)}\smsp.
\]
\end{subparagr}

\begin{subparagr} \label{bincompll}
\textbf{Level $l$ binary composition.} The above construction can be iterated in order to obtain by induction on $l\geqslant2$, for every $i\geqslant l$, a map $\comultbin[l]{i}:\Dn{i}\to\Dn{i}\amalg_{\Dn{i-l}}\Dn{i}$, lifting of the pair of parallel arrows $\bigl((\Ths{i}\amalg_{\Dn{i-l}}\Ths{i})\comultbin[l-1]{i-1},(\Tht{i}\amalg_{\Dn{i-l}}\Tht{i})\comultbin[l-1]{i-1}\bigr)$, defining a map
\[
\xymatrixrowsep{2.7pc}
\xymatrixcolsep{2.7pc}
\xymatrix{
\GrpFl{i}\times_{\GrpFl{i-l}}\GrpFl{i}\simeq\Grp(\Dn{i}\amalg_{\Dn{i-l}}\Dn{i}) \ar[r]^-{\Grp(\comultbin[l]{i})}
&\Grp(\Dn{i})=\GrpFl{i}
},
\]
associating to each pair of $i$-cells $(x,y)$ such that the iterated source of $x$ in $\GrpFl{i-l}$ is equal to the iterated target of $y$, a $i$\nobreakdash-cell (level $l$ composition of $x$ and $y$)
\[
\multbin[l]{i}{x}{y}:\,=\Grp(\comultbin[l]{i})(x,y)
\]
such that
\[
\Grps{i}(\multbin[l]{i}{x}{y})=\multbin[l-1]{i-1}{\Grps{i}(x)}{\Grps{i}(y)}\hhbox{and}\Grpt{i}(\multbin[l]{i}{x}{y})=\multbin[l-1]{i-1}{\Grpt{i}(x)}{\Grpt{i}(y)}\smsp.
\]
\end{subparagr}

\begin{subparagr} \label{mcompl1}
\textbf{Level $1$ $m$-ary composition.} There are many more general compositions in the structure of the $\infty$\nobreakdash-$\Coh$\nobreakdash-groupoid $\Grp$. For example, let $m$ be an integer, $m\geqslant2$. For every $i\geqslant1$, the two composite arrows
\[
\xymatrixcolsep{2.7pc}
\xymatrixrowsep{.7pc}
\xymatrix{
\Dn{i-1}\ar[r]^-{\Ths{i}}
&\Dn{i}\ar[r]^-{\can{m}}
&\Dn{i}\amalg_{\Dn{i-1}}\cdots\amalg_{\Dn{i-1}}\Dn{i}\smsp,
\\
\Dn{i-1}\ar[r]^-{\Tht{i}}
&\Dn{i}\ar[r]^-{\can{1}}
&\Dn{i}\amalg_{\Dn{i-1}}\cdots\amalg_{\Dn{i-1}}\Dn{i}\smsp,
}
\]
with target the globular sum of length $m$, form a pair of parallel arrows in $\Coh$, hence a lifting
\[
\comultbin[1]{i,m}:\Dn{i}\toto\Dn{i}\amalg_{\Dn{i-1}}\cdots\amalg_{\Dn{i-1}}\Dn{i}\smsp,
\]
inducing a map
\[
\xymatrixcolsep{5.7pc}
\xymatrix{
\GrpFl{i}\times_{\GrpFl{i-1}}\cdots\times_{\GrpFl{i-1}}\GrpFl{i}\ar[r]^-{\multbin[1]{i,m}{}{}:\,=\Grp(\comultbin[1]{i,m})}
&\GrpFl{i}\smsp.
}
\]
This map defines a $m$-ary composition, associating to each ``composable'' $m$-uple $(x_1,\dots,x_m)$ of $i$\nobreakdash-cells, a $i$\nobreakdash-cell 
\[
\multbin[1]{i,m}{}{}(x_1,\dots,x_m)=\Grp(\comultbin[1]{i,m})(x_1,\dots,x_m)
\]
such that
\[
\Grps{i}\bigl(\,\multbin[1]{i,m}{}{}(x_1,\dots,x_m)\bigr)=\Grps{i}(x_m)\hhbox{and}
\Grpt{i}\bigl(\,\multbin[1]{i,m}{}{}(x_1,\dots,x_m)\bigr)=\Grpt{i}(x_1)\smsp.
\]
\end{subparagr}

\begin{subparagr} \label{assbincompl1}
\textbf{Associativity constraint for level $1$ binary composition.}  For every $i\geqslant1$, observe that the two composite arrows
\[
\xymatrixcolsep{4.7pc}
\xymatrixrowsep{.7pc}
\xymatrix{
\Dn{i}\ar[r]^-{\comultbin{i}}
&\Dn{i}\amalg_{\Dn{i-1}}\Dn{i}\ar[r]^-{\comultbin{i}\amalg_{\Dn{i-1}}\id{\Dn{i}}}
&\Dn{i}\amalg_{\Dn{i-1}}\Dn{i}\amalg_{\Dn{i-1}}\Dn{i}
\\
\Dn{i}\ar[r]^-{\comultbin{i}}
&\Dn{i}\amalg_{\Dn{i-1}}\Dn{i}\ar[r]^-{\id{\Dn{i}}\amalg_{\Dn{i-1}}\comultbin{i}}
&\Dn{i}\amalg_{\Dn{i-1}}\Dn{i}\amalg_{\Dn{i-1}}\Dn{i}
},
\]
(where $\comultbin{i}=\comultbin[1]{i}$ is as in \ref{bincompl1}) form a pair of parallel arrows in $\Coh$, therefore there exists a lifting 
\[
\coass{}:\,=\coass{i}:\,=\coass[1]{i}:\Dn{i+1}\toto\Dn{i}\amalg_{\Dn{i-1}}\Dn{i}\amalg_{\Dn{i-1}}\Dn{i}\smsp.
\]
We deduce a map 
\[
\xymatrixcolsep{5.7pc}
\xymatrix{
\GrpFl{i}\times_{\GrpFl{i-1}}\GrpFl{i}\times_{\GrpFl{i-1}}\GrpFl{i} \ar[r]^-{\ass{}:\,=\ass[1]{i}:\,=\Grp(\coass[1]{i})}
&\GrpFl{i+1}
},
\]
associating to each triple of ``composable'' $i$-cells $(x,y,z)$ an \emph{associativity constraint} $(i+1)$\nobreakdash-cell $\ass{}{}_{x,y,z}$ such that (in the notations of~\ref{bincompl1})
\[
\Grps{i+1}(\ass{}{}_{x,y,z})=\multbin{}{(\multbin{}{x}{y})}{z}\hhbox{and}\Grpt{i+1}(\ass{}{}_{x,y,z})=\multbin{}{x}{(\multbin{}{y}{z})}\smsp.
\]
\end{subparagr}

\begin{subparagr} \label{assbincompl2}
\textbf{Associativity constraint for level $2$ binary composition.} The construction of associativity constraints for higher-level compositions becomes more complicated. For example, for the level $2$ composition, observe (in the notations of~\ref{bincompl2}) that 
\[
\bigl((\comultbin[2]{i}\amalg_{\Dn{i-2}}\id{\Dn{i}})\comultbin[2]{i},\,(\id{\Dn{i}}\amalg_{\Dn{i-2}}\comultbin[2]{i})  \comultbin[2]{i}\bigr)\smsp,\quad i\geqslant2\smsp,
\]
is \emph{not} a pair of parallel arrows as
\[
\begin{aligned}
(\comultbin[2]{i}\amalg_{\Dn{i-2}}\id{\Dn{i}})\comultbin[2]{i}\Ths{i}&=(\Ths{i}\amalg_{\Dn{i-2}}\Ths{i}\amalg_{\Dn{i-2}}\Ths{i})(\comultbin[1]{i-1}\amalg_{\Dn{i-2}}\id{\Dn{i-1}})\comultbin[1]{i-1}\cr
\noalign{\vskip 3pt}
&\kern -80pt\neq(\Ths{i}\amalg_{\Dn{i-2}}\Ths{i}\amalg_{\Dn{i-2}}\Ths{i})(\id{\Dn{i-1}}\amalg_{\Dn{i-2}}\comultbin[1]{i-1})\comultbin[1]{i-1}=(\id{\Dn{i}}\amalg_{\Dn{i-2}}\comultbin[2]{i})\comultbin[2]{i}\Ths{i}
\end{aligned}
\]
and similarly with $\Ths{i}$ replaced by $\Tht{i}$. In order to be able to define an associativity constraint for level 2 composition, verify that the two following composite arrows
\[
\xymatrixcolsep{1.63pc}
\xymatrixrowsep{.7pc}
\xymatrix{
\Dn{i}\ar[r]^-{\comultbin[1]{i}}
&\Dn{i}\!\amalg_{\Dn{i-1}}\!\!\Dn{i}\ar[rrrrrrr]^-{(( \Tht{i}\amalg_{\Dn{i-2}}\Tht{i}\amalg_{\Dn{i-2}}\Tht{i})\coass[1]{i-1}\kern -2pt,\,(\comultbin[2]{i}\amalg_{\Dn{i-2}}\id{\Dn{i}})\comultbin[2]{i})}
&&&&&&&\Dn{i}\!\amalg_{\Dn{i-2}}\kern -3pt\Dn{i}\!\amalg_{\Dn{i-2}}\kern -3pt\Dn{i},
\\
\Dn{i}\ar[r]^-{\comultbin[1]{i}}
&\Dn{i}\!\amalg_{\Dn{i-1}}\!\!\Dn{i}\ar[rrrrrrr]^-{((\id{\Dn{i}} \amalg_{\Dn{i-2}}  \comultbin[2]{i} )\comultbin[2]{i},\,(\Ths{i}\amalg_{\Dn{i-2}}\Ths{i}\amalg_{\Dn{i-2}}\Ths{i})\coass[1]{i-1})}
&&&&&&&\Dn{i}\!\amalg_{\Dn{i-2}}\kern -3pt\Dn{i}\!\amalg_{\Dn{i-2}}\kern -3pt\Dn{i}
}
\]
form a pair of parallel arrows. Hence a lifting
\[
\coass[2]{i}:\Dn{i+1}\toto\Dn{i}\amalg_{\Dn{i-2}}\Dn{i}\amalg_{\Dn{i-2}}\Dn{i}\smsp,
\]
inducing a map 
\[
\xymatrixcolsep{4.3pc}
\xymatrix{
\GrpFl{i}\times_{\GrpFl{i-2}}\GrpFl{i}\times_{\GrpFl{i-2}}\GrpFl{i} \ar[r]^-{\ass[2]{i}:\,=\Grp(\coass[2]{i})}
&\GrpFl{i+1}
},
\]
associating to each triple of $i$-cells $(x,y,z)$ ``composable'' over $\GrpFl{i-2}$ an \emph{associativity constraint} $(i+1)$\nobreakdash-cell $\ass[2;]{i}{}_{x,y,z}$ such that (in the notations of~\ref{bincompl1},~\ref{bincompl2} and~\ref{assbincompl1})
\[
\begin{aligned}
\Grps{i+1}(\ass[2;]{i}{}_{x,y,z})&=\multbin[1]{i}{\ass[1;\,\Grpt{i}(x),\Grpt{i}(y),\Grpt{i}(z)]{i-1}\,}{\,\Bigl(\multbin[2]{i}{(\multbin[2]{i}{x}{y})}{z}\Bigr)}\smsp,\cr
\Grpt{i+1}(\ass[2;]{i}{}_{x,y,z})&=\multbin[1]{i}{\Bigl( \multbin[2]{i}{x}{(\multbin[2]{i}{y}{z})} \Bigr)\,}{\,\ass[1;\,\Grps{i}(x),\Grps{i}(y),\Grps{i}(z)]{i-1}} \smsp.
\end{aligned}
\]
It is left as an exercise to the reader to proceed to the construction of associativity constraints for higher-level compositions.
\end{subparagr}

\begin{subparagr}
\textbf{Pentagon and exchange constraints.} Similarly, Mac Lane's pentagon and Godement's exchange rule give rise to ``higher'' constraints, defined by choosing suitable pairs of parallel arrows (see~\cite[4.2.4]{Ara}). 
\end{subparagr}

\begin{subparagr} \label{units}
\textbf{Units.} Let $i\geqslant0$. The most trivial pair of parallel arrows is $(\id{\Dn{i}},\id{\Dn{i}})$. It gives rise to a lifting $\counit{i}$ such that
\[
\xymatrixrowsep{2.7pc}
\xymatrixcolsep{4.7pc}
\xymatrix{
\Dn{i+1}\ar@<.9ex>[rd]^{\counit{i}}
&
& \ar@{}@<-7ex>[r]^{\textstyle \counit{i}\Ths{i+1}=\id{\Dn{i}}=\counit{i}\Tht{i+1}\smsp.} 
&
\\
\,\Dn{i} \ar@<1ex>[u]^{\Ths{i+1}} \ar@<-1ex>[u]_{\Tht{i+1}}
           \ar@<1ex>[r]^-{\id{\Dn{i}}} \ar@<-1ex>[r]_-{\id{\Dn{i}}}
&\Dn{i}
}
\]
It defines a map 
\[
\unit{i}:\,=\Grp(\counit{i}):\GrpFl{i}\toto\GrpFl{i+1}\smsp,
\] 
associating to each $i$\nobreakdash-cell $x$ a \emph{unit} $(i+1)$\nobreakdash-cell $\varunit{x}:\,=\unit{i}(x)$ such that
\[
\Grps{i+1}(\varunit{x})=x=\Grpt{i+1}(\varunit{x})\smsp
\]
(the name of unit and the notation being justified by what follows).
\end{subparagr}

\begin{subparagr} \label{unitcons}
\textbf{Unit constraints.} 
Let $i\geqslant1$. Observe (using the notations of~\ref{bincompl1} and~\ref{units}) that
\[
\bigl((\Tht{i}\counit{i-1},\id{\Dn{i}})\comultbin{i}, \id{\Dn{i}}\bigr)\hhbox{and}
\bigl((\id{\Dn{i}},\Ths{i}\counit{i-1})\comultbin{i},\id{\Dn{i}}\bigr)
\]
are pairs of parallel arrows, hence liftings $\counitlcontr{i},\counitrcontr{i}:\Dn{i+1}\to\Dn{i}$ such that
\[
\begin{aligned}
\counitlcontr{i}\Ths{i+1}=(\Tht{i}\counit{i-1},\id{\Dn{i}})\comultbin{i}\smsp,\qquad
\counitlcontr{i}\Tht{i+1}=\id{\Dn{i}}\smsp,\cr
\noalign{\vskip 3pt}
\counitrcontr{i}\Ths{i+1}=(\id{\Dn{i}},\Ths{i}\counit{i-1})\comultbin{i}\smsp,\qquad
\counitrcontr{i}\Tht{i+1}=\id{\Dn{i}}\smsp.
\end{aligned}
\]
We deduce maps
\[
\xymatrixcolsep{4pc}
\xymatrix{
\GrpFl{i}\ar[r]^-{\unitlcontr{i}:\,=\Grp(\counitlcontr{i})}
&\GrpFl{i+1}
}
\smsp,\quad
\xymatrixcolsep{4pc}
\xymatrix{
\GrpFl{i}\ar[r]^-{\unitrcontr{i}:\,=\Grp(\counitrcontr{i})}
&\GrpFl{i+1}
}
\]
associating to each $i$\nobreakdash-cell $x$ a \emph{left}, respectively \emph{right}, \emph{unit constraint} $(i+1)$\nobreakdash-cell $\unitlcontr{i}(x)$, respectively $\unitrcontr{i}(x)$, such that (in the notations of~\ref{bincompl1} and~\ref{units})
\[
\begin{aligned}
\Grps{i+1}(\unitlcontr{i}(x))&=\multbin[1]{i}{\varunit{\Grpt{i}(x)}\kern 1pt}{\kern 1ptx}\smsp,\qquad
\Grpt{i+1}(\unitlcontr{i}(x))=x\smsp,\cr
\Grps{i+1}(\unitrcontr{i}(x))&=\multbin[1]{i}{x\kern 1pt}{\kern 1pt\varunit{\Grps{i}(x)}}\smsp,\qquad
\Grpt{i+1}(\unitrcontr{i}(x))=x\smsp.
\end{aligned}
\]
\end{subparagr}

\begin{subparagr} \label{trianglecons}
\textbf{Triangle constraint.} In a similar way a \emph{triangle} higher constraint can be defined, involving associativity, left and right unit constraints (see~\cite[4.2.4]{Ara}).
\end{subparagr}

\begin{subparagr} \label{invl1}
\textbf{Level 1 inverse.} 
For every $i\geqslant1$, $(\Tht{i},\Ths{i})$ is a pair of parallel arrows, hence a lifting $\coinv{i}:\,=\coinv[1]{i}:\Dn{i}\to\Dn{i}$ such that
\[
\coinv{i}\Ths{i}=\Tht{i}\hhbox{and}\coinv{i}\Tht{i}=\Ths{i}\smsp.
\]
We deduce a map $\inv{i}:\,=\inv[1]{i}:\,=\Grp(\coinv[1]{i}):\GrpFl{i}\to\GrpFl{i}$, associating to a $i$\nobreakdash-cell an ``\emph{inverse}'' $i$\nobreakdash-cell $\varinv{x}:\,=\inv[1]{i}(x)$ such that
\[
\Grps{i}(\varinv{x})=\Grpt{i}(x)\hhbox{and}\Grpt{i}(\varinv{x})=\Grps{i}(x)\smsp.
\]
(the name of ``inverse'' and the notation being justified by what follows).
\end{subparagr}

\begin{subparagr}
\textbf{Level 1 inverse constraints.} \label{invconsl1}
Let $i\geqslant1$. Observe (using the notations of~\ref{bincompl1}, \ref{units} and~\ref{invl1}) that
\[
\bigl((\coinv{i},\id{\Dn{i}})\comultbin{i},\Ths{i}\counit{i-1}\bigr)\hhbox{and}
\bigl((\id{\Dn{i}},\coinv{i})\comultbin{i},\Tht{i}\counit{i-1}\bigr)
\]
are pairs of parallel arrows, hence liftings $\coinvlcontr{i},\coinvrcontr{i}:\Dn{i+1}\to\Dn{i}$ such that
\[
\begin{aligned}
\coinvlcontr{i}\Ths{i+1}=(\coinv{i},\id{\Dn{i}})\comultbin{i}\smsp,\qquad
\coinvlcontr{i}\Tht{i+1}=\Ths{i}\counit{i-1}\smsp,\cr
\noalign{\vskip 3pt}
\coinvrcontr{i}\Ths{i+1}=(\id{\Dn{i}},\coinv{i})\comultbin{i}\smsp,\qquad
\coinvrcontr{i}\Tht{i+1}=\Tht{i}\counit{i-1}\smsp.
\end{aligned}
\]
We deduce maps
\[
\xymatrixcolsep{4pc}
\xymatrix{
\GrpFl{i}\ar[r]^-{\invlcontr{i}:\,=\Grp(\coinvlcontr{i})}
&\GrpFl{i+1}
},
\,\quad
\xymatrixcolsep{4pc}
\xymatrix{
\GrpFl{i}\ar[r]^-{\invrcontr{i}:\,=\Grp(\coinvrcontr{i})}
&\GrpFl{i+1}
},
\]
associating to each $i$\nobreakdash-cell $x$ a \emph{left}, respectively \emph{right}, \emph{inverse constraint} $(i+1)$\nobreakdash-cell $\invlcontr{i}(x)$, respectively $\invrcontr{i}(x)$, such that (in the notations of~\ref{bincompl1}, \ref{units} and~\ref{invl1})
\[
\begin{aligned}
\Grps{i+1}(\invlcontr{i}(x))&=\multbin[1]{i}{\varinv{x}}{\kern 2pt x}\smsp,\qquad
\Grpt{i+1}(\invlcontr{i}(x))=\varunit{\Grps{i}(x)}\smsp,\cr
\Grps{i+1}(\invrcontr{i}(x))&=\multbin[1]{i}{x\kern 1pt}{\kern 1pt\varinv{x}}\smsp,\qquad
\Grpt{i+1}(\invrcontr{i}(x))=\varunit{\Grpt{i}(x)}\smsp.
\end{aligned}
\]
\end{subparagr}

\begin{subparagr} \label{invl2}
\textbf{Level 2 inverse.} 
Let $i\geqslant2$. Observe (using the notations of~\ref{invl1}) that $(\Ths{i}\coinv[1]{i-1},\Tht{i}\coinv[1]{i-1})$ is a pair of parallel arrows, hence a lifting $\coinv[2]{i}:\Dn{i}\to\Dn{i}$ defining a map $\inv[2]{i}:\,=\Grp(\coinv[2]{i}):\GrpFl{i}\to\GrpFl{i}$, associating to a $i$\nobreakdash-cell $x$ a ``level 2 inverse'' $i$\nobreakdash-cell $\inv[2]{i}(x)$ such that
\[
\Grps{i}(\inv[2]{i}(x))=\inv[1]{i-1}(\Grps{i}(x))\hhbox{and}
\Grpt{i}(\inv[2]{i}(x))=\inv[1]{i-1}(\Grpt{i}(x))\smsp.
\]
\end{subparagr}

\begin{subparagr} \label{invll}
\textbf{Level $l$ inverse.} The above construction can be iterated in order to obtain by induction on $l\geqslant2$, for every $i\geqslant l$, a map $\coinv[l]{i}:\Dn{i}\to\Dn{i}$, lifting of the pair of parallel arrows $(\Ths{i}\coinv[l-1]{i-1},\Tht{i}\coinv[l-1]{i-1})$, defining a map $\inv[l]{i}:\,=\Grp(\coinv[l]{i}):\GrpFl{i}\to\GrpFl{i}$, associating to a $i$\nobreakdash-cell $x$ a ``level $l$ inverse'' $i$\nobreakdash-cell $\inv[l]{i}(x)$ such that
\[
\Grps{i}(\inv[l]{i}(x))=\inv[l-1]{i-1}(\Grps{i}(x))\hhbox{and}
\Grpt{i}(\inv[l]{i}(x))=\inv[l-1]{i-1}(\Grpt{i}(x))\smsp.
\]
It is left as an exercise to the reader to define level $l$ inverse constraints relating level $l$ inverse with level $l$ binary composition.
\end{subparagr}
\end{paragr}

\section{Grothendieck's conjecture.}
\begin{paragr} \label{homotopygr}
\textbf{The homotopy groups of a $\infty$-groupoid.} Fix a \Gr-coherator $\Coh$ and a $\infty$\nobreakdash-$\Coh$\nobreakdash-groupoid \hbox{$\Grp:\Coh^\op\to\Ens$}. We will freely use the notations of the structural maps introduced in the previous section.
\smallbreak

For every $i\geqslant0$, we define a \emph{homotopy} relation $\hmtp{i}$ between $i$\nobreakdash-cells of $\Grp$ by
\[
x\hmtp{i}y\quad \buildrel\hbox{\scriptsize def}\over{\Longleftrightarrow}\quad \exists h\in \GrpFl{i+1}\quad \Grps{i+1}(h)=x\smsp,\quad \Grpt{i+1}(h)=y\smsp.
\]
The homotopy relation is an equivalence relation:
\smallbreak

\emph{a}) \textbf{Reflexivity.} Let $x$ be a $i$-cell. We have (\cf~\ref{units}):
\[
\Grps{i+1}(\varunit{x})=x\hhbox{and}\Grpt{i+1}(\varunit{x})=x
\quad,
\]
therefore $x\hmtp{i}x$.

\emph{b}) \textbf{Symmetry.} Let $x,\,y$ be two $i$-cells such that $x\hmtp{i} y$. By definition, there exists a $(i+1)$\nobreakdash-cell $h$ such that $\Grps{i+1}(h)=x$ and $\Grpt{i+1}(h)=y$, therefore (\cf~\ref{invl1})
\[
\Grps{i+1}(\varinv{h})=\Grpt{i+1}(h)=y\hhbox{and}\Grpt{i+1}(\varinv{h})=\Grps{i+1}(h)=x\smsp,
\]  
and hence $y\hmtp{i} x$. 
\smallbreak

\emph{c}) \textbf{Transitivity.} Let $x,\,y,\,z$ be three $i$-cells such that $x\hmtp{i}y$ and $y\hmtp{i}z$. By definition, there exist two $(i+1)$\nobreakdash-cells $h,k$ such that 
\[
\Grps{i+1}(h)=x\smsp ,\quad
\Grpt{i+1}(h)=y\smsp,\quad
\Grps{i+1}(k)=y\smsp,\quad
\Grpt{i+1}(k)=z\smsp.\quad
\]
In particular, $(k,h)$ is an element of the (globular) fiber product $\GrpFl{i+1}\times_{\GrpFl{i}}\GrpFl{i+1}$, so $\multbin{}{k}{h}$ is defined and (\cf~\ref{bincompl1})
\[
\Grps{i+1}(\multbin{}{k}{h})=\Grps{i+1}(h)=x\hhbox{and}\Grpt{i+1}(\multbin{}{k}{h})=\Grpt{i+1}(k)=z\smsp,
\]
which proves that $x\hmtp{i}z$.
\smallbreak
\noindent Two $i$-cells $x,\,y$ are called \emph{homotopic} if $x\hmtp{i}y$. We denote by $\GrpFlQ{i}$ the quotient of the set $\GrpFl{i}$ of $i$\nobreakdash-cells by the homotopy equivalence relation $\hmtp{i}$. We define the \emph{set of connected components} of the $\infty$\nobreakdash-$\Coh$\nobreakdash-groupoid $\Grp$ as the set $\Pizero{\Grp}:\,=\GrpFlQ{0}:\,=\GrpFl{0}/\hmtp{0}$.
\smallbreak

Suppose now that $i\geqslant1$ and observe that if $x,\,y$ are two homotopic $i$\nobreakdash-cells, then the globular relations imply that $\Grps{i}(x)=\Grps{i}(y)$ and $\Grpt{i}(x)=\Grpt{i}(y)$. Therefore the maps $\Grps{i},\Grpt{i}:\GrpFl{i}\to\GrpFl{i-1}$ induce maps $\GrpsQ{i},\GrptQ{i}:\GrpFlQ{i}\to\GrpFl{i-1}$. On the other hand, the equivalence relation $\hmtp{i}$ is compatible with the composition \smash{$\multbin{}{}{}=\multbin[1]{i}{}{}$}. Let us prove for example that if $x_1,\,x_2$ are two homotopic $i$\nobreakdash-cells, then for every $i$\nobreakdash-cell $y$ with target the common source of $x_1$ and $x_2$, the $i$\nobreakdash-cells $\multbin{}{x_1}{y}$ and $\multbin{}{x_2}{y}$ are homotopic. By definition, there exists a $(i+1)$\nobreakdash-cell $h$ such that $\Grps{i+1}(h)=x_1$ and $\Grpt{i+1}(h)=x_2$. If we denote by $h'$ the $(i+1)$\nobreakdash-cell 
\[
h'=\multbin[2]{i+1}{h}{\varunit{y}}\smsp, 
\]
then we have (\cf~\ref{bincompl2} and~\ref{units})
\[
\Grps{i+1}(h')= \multbin[1]{i}{\Grps{i+1}(h)}{\Grps{i+1}(\varunit{y})}=\multbin[1]{i}{x_1}{y}
\ \ \hbox{and}\ \
\Grpt{i+1}(h')= \multbin[1]{i}{\Grpt{i+1}(h)}{\Grpt{i+1}(\varunit{y})}=\multbin[1]{i}{x_2}{y}
\smsp,
\]
and hence $\multbin{}{x_1}{y}$ and $\multbin{}{x_2}{y}$ are homotopic. Therefore the map \smash{$\multbin[1]{i}{}{}:\GrpFl{i}\times_{\GrpFl{i-1}}\GrpFl{i}\to\GrpFl{i}$} induces a map $\GrpFlQ{i}\times_{\GrpFl{i-1}}\GrpFlQ{i}\to\GrpFlQ{i}$. The existence of the associativity constraint and the unit constraints (\cf~\ref{assbincompl1} and~\ref{unitcons}) implies that this composition defines a category $\varpi_i(\Grp)$ with object set $\GrpFl{i-1}$, arrow set $\GrpFlQ{i}$, and target and source maps $\GrpsQ{i},\,\GrptQ{i}$. The existence of the inverse constraints (\cf~\ref{invconsl1}) implies that this category is a groupoid. The remark at the end of~\ref{bincompl1} proves that this groupoid is independent of the choice of the lifting $\comultbin[1]{i}$ defining the composition \smash{$\multbin[1]{i}{}{}$}. For more details, see~\cite{Ara}, propositions~4.3.2, 4.3.3.
\smallbreak

Let $x$ be a $0$-cell of $\Grp$, $x\in\GrpFl{0}$, and $i$ an integer, $i\geqslant1$. The $i$-th \emph{homotopy group of $\Grp$ at $x$} is the group 
\[
\pi_i(\Grp;x):\,=\Hom_{\varpi_i(\Grp)}(\unit{}(x),\unit{}(x))\smsp,
\]
where using the notations of~\ref{units}, $\unit{}=\unit{i-2}\dots\unit{1}\unit{0}$. In order to justify this\break definition, it should be verified that this group is, up to canonical isomorphism, independent of the choice of the lifting arrows $\counit{j}$, $0\leqslant j\leqslant i-2$, giving rise to the maps~$\unit{j}$.
\end{paragr}

\begin{paragr} \label{infeq}
\textbf{Weak equivalences of $\infty$-groupoids.} Fix a \Gr-coherator $\Coh$. A morphism $\GrpMor:\Grp\to\Grp'$ of $\infty$\nobreakdash-$\Coh$\nobreakdash-groupoids is called  a \emph{weak equivalence} or a $\infty$\nobreakdash-\emph{equivalence} if the following two conditions are satisfied:
\begin{itemize}
\item[\emph{a})] the map $\Pizero{\GrpMor}:\Pizero{\Grp}\to\Pizero{\Grp'}$, induced by $\GrpMor$, is a bijection;
\item[\emph{b})] for every $i\geqslant1$ and every $0$\nobreakdash-cell $x$ of $\Grp$, the map 
\[
\pi_i(\GrpMor;x):\pi_i(\Grp;x)\to\pi_i(\Grp';\GrpMor(x))\smsp,
\] 
induced by $\GrpMor$, is an isomorphism of groups.
\end{itemize}
\end{paragr}

\begin{paragr} \label{Grothconjfaible}
\textbf{Grothendieck's conjecture (weak form).} 
\emph{For every \Gr-coherator $\Coh$, the localization of the category of $\infty$\nobreakdash-$\Coh$\nobreakdash-groupoids by the $\infty$\nobreakdash-equivalences is equivalent to the homotopy category of $\mathsf{CW}$\nobreakdash-complexes.}

For a strategy for proving this conjecture, see~\cite{Ma1}. A more precise form of this conjecture is given below (\cf~\ref{Grothconjfort}).
\end{paragr}

\begin{paragr} \label{topglob}
\textbf{The topological realization of the category of globes.}
For $i\geqslant0$, let $\TopD{i}$ be the $i$-dimensional topological disk
\[
\TopD{i}=\{x\in\mathbb{R}^i\mid\,\Vert x\Vert\leq1\}\smsp,
\]
where $\Vert x\Vert$ denotes the Euclidian norm of $x$, and for $i>0$,
\[
\Tops{i},\Topt{i}:\TopD{i-1}\,{{\raise 1.5pt\vbox{\xymatrixcolsep{1.5pc}\xymatrix{\ar@{^{(}->}[r]&}}\hskip -2pt}}\,\TopD{i}\smsp
\]
the inclusions defined by 
\[
\Tops{i}(x)= \bigl(x,\,-\sqrt{1-\Vert x\Vert^2}\,\bigr)
\smsp,\qquad
\Topt{i}(x)= \bigl(x,\,\sqrt{1-\Vert x\Vert^2}\,\bigr)
\smsp,\qquad x\in \TopD{i-1}\smsp.
\]
The maps $\Tops{i}$ and $\Topt{i}$ factorize through the $(i-1)$\nobreakdash-dimensional sphere, boundary of~$\TopD{i}$,
\[
\Sphere{i-1}=\bordD{i}=\{x\in\mathbb{R}^i\mid\,\Vert x\Vert=1\}\smsp,
\]
identifying $\TopD{i-1}$ to the lower respectively upper hemisphere of $\Sphere{i-1}$. As the maps $\Tops{i}$, $\Topt{i}$, $i>0$, satisfy the coglobular relations, the assignment 
\[
\Dn{i}\longmapsto\TopD{i}\smsp,\quad\Ths{i}\longmapsto\Tops{i}\smsp,\quad\Tht{i}\longmapsto\Topt{i}\smsp
\]
defines a functor $\Glob\to\Top$ from the category of globes (\cf~\ref{catglob}) to the category of topological spaces. It is easy to verify that this functor is faithful, identifying $\Glob$ to the (non-full) subcategory of $\Top$ with objects the disks $\TopD{i}$, $i\geqslant0$, and morphisms generated by the inclusions $\Tops{i}$, $\Topt{i}$, $i>0$.
\end{paragr}

\begin{paragr} \label{topglobext}
\textbf{Topological spaces as a globular extension.}
As the category $\Top$ of topological spaces is cocomplete, $\Top$ endowed with the functor $\Glob\to\Top$, defined above, is a globular extension (\cf~\ref{extglob}). By the universal property of $\Thetazero$, there exists a morphism of globular extensions $\Thetazero\to\Top$, unique up to unique natural isomorphism, extending the functor $\Glob\to\Top$. It is easy to verify that this functor is faithful, identifying $\Thetazero$ to the (non-full) subcategory of $\Top$ with objects globular sums of disks and morphisms continuous maps 
\[
\xymatrixcolsep{1.8pc}
\xymatrix{
X=\TopD{i_1}\!\amalg^{}_{\TopD{i'_1}}\!\!\TopD{i_2}\!\amalg^{}_{\TopD{i'_2}}\!\cdots\,\,\amalg^{}_{\TopD{i'_{m-1}}}\!\!\!\!\TopD{i_m}
\ar@<1ex>[r]^-{\varphi}
&\TopD{j_1}\!\amalg^{}_{\TopD{j'_1}}\!\!\TopD{j_2}\!\amalg^{}_{\TopD{j'_2}}\!\cdots\,\,\amalg^{}_{\TopD{j'_{n-1}}}\!\!\!\!\TopD{j_n}=Y
}
\]
such that for every $k$, $1\leqslant k\leqslant m$, there exists an integer $l$, $1\leqslant l\leqslant n$, and a commutative square
\[
\xymatrix{
&\TopD{i_k}\ar[r]^{\psi} \ar[d]_{\mathrm{can}^{}_k}
&\TopD{i_l}\ar[d]^{\mathrm{can}^{}_l}
\\
&X\ar[r]_{\varphi}
&Y
&\hskip -30pt,
}\hskip 10pt
\]
with $\psi$ the image of an arrow of $\Glob$, \ie~a composite of $\Tops{i}$'s or $\Topt{i}$'s. All objects of this subcategory are contractible spaces.
\end{paragr}

\begin{paragr} \label{fondinfgrpd}
\textbf{\boldmath The fundamental $\infty$-groupoid of a space.}
Let $\Coh$ be a \Gr-coherator. In order to associate functorially to every topological space a $\infty$\nobreakdash-$\Coh$\nobreakdash-groupoid, it's enough to define a $\infty$\nobreakdash-\emph{cogroupoid} of type $\Coh$ in $\Top$, \ie~a functor
\[
\Coh\toto\Top
\]
preserving globular sums. The \emph{fundamental} $\infty$\nobreakdash-\emph{groupoid} of a topological space $X$ can then be defined as the composite of the opposite functor
\[
\Coh^\op\toto\Top^\op
\]
with the representable presheaf on $\Top$ defined by $X$
\[
\Hom^{}_{\Top}(\,?\,,X):\Top^\op\toto\Ens\smsp.
\]
Let
\[
\xymatrixcolsep{1.4pc}
\xymatrix{
\Glob \ar[r]
&\Thetazero=\Coh_0\ar[r]
&\Coh_1\ar[r]
&\cdots\ar[r]
&\Coh_n\ar[r]
&\Coh_{n+1}\ar[r]
&\cdots\ar[r]
&\Coh\simeq\varinjlim\Coh_n
}\smsp
\]
be a tower of definition of the \Gr\nobreakdash-coherator $\Coh$. We have already defined a functor $\Thetazero\to\Top$ preserving globular sums. In order to prove that this functor can be extended to a functor $\Coh\to\Top$ preserving globular sums, \ie~to a morphism of globular extensions, it's enough to prove that for every $n\geqslant0$, any morphism of globular extensions $\Coh_n\to\Top$ can be extended to a morphism of globular extensions $\Coh_{n+1}\to\Top$. Using the universal property of $\Coh_{n+1}$, it's enough to prove that for every globular sum $X$ in $\Top$, every $i\geqslant0$, and every pair $(f,g):\TopD{i}\to X$ of continuous maps such that if $i>0$, we have
\[
f\Tops{i}=g\Tops{i}\hhbox{and}f\Topt{i}=g\Topt{i}\smsp,
\]
there exists a continuous map $h:\TopD{i+1}\to X$ such that
\[
f=h\Tops{i+1}\hhbox{and}g=h\Topt{i+1}\smsp.
\]
\[
\xymatrixcolsep{7pc}
\xymatrixrowsep{3.3pc}
\xymatrix{
\ \TopD{i+1}\ar@<1ex>@{-->}[dr]^h
\\
\TopD{i}\ar@<1ex>[r]^f\ar@<-1ex>[r]_g
\ar@<1ex>[u]^{\Tops{i+1}}\ar@<-1ex>[u]_{\Topt{i+1}}
&X
}
\]
As $X$ is a contractible space, the map $X\to *$ of $X$ to the point is a trivial fibration. Therefore, as the inclusion $\Sphere{i}\hookrightarrow \TopD{i+1}$ is a cofibration, the existence of such a map $h$ is a consequence of the lifting  property of the following square
$$
\xymatrixcolsep{2pc}
\xymatrixrowsep{2.7pc}
\xymatrix{
\Sphere{i}=\TopD{i}\amalg_{\Sphere{i-1}}\TopD{i}\ar@{^{(}->}[d]_{(\Tops{i+1},\Topt{i+1})}\ar[rr]^(.65){(f,g)}
&&X\ar[d]
\\
\ \TopD{i+1}\ar@{-->}[rru]^h\ar[rr]
&&\hbox{*}
&.
}
$$
The functor $\Coh\to\Top$ defined this way is \emph{not} independent of the choice of the relevant lifting maps $h$. Nevertheless, it is conjectured that this dependence is inessential (\cf~\ref{Grothconjfort}). As the functor $\Thetazero\to\Top$ is faithful, the existence of the extension $\Coh\to\Top$ implies that the functor $\Thetazero\to\Coh$ is faithful. It is not known if the functors $\Coh_n\to\Top$ extending $\Thetazero\to\Top$ can be chosen faithful, which would imply the conjecture that the functors $\Coh_n\to\Coh$ are faithful (\cf~\ref{cohgr}).
\end{paragr}

\begin{paragr}
\textbf{\boldmath The classifying space of a $\infty$-groupoid.} 
Fix a \Gr-coherator $\Coh$ and choose, as above, an extension $\Coh\to\Top$ of the canonical functor $\Thetazero\to\Top$, preserving globular sums. As the category $\Top$ is cocomplete, this functor induces a pair of adjoint functors
\[
\pref{\Coh}\toto\Top\smsp,\qquad\Top\toto\pref{\Coh}\smsp
\]
between the category of presheaves on $\Coh$ and the category of topological spaces. The functor $\Top\to\pref{\Coh}$ factors through the full subcategory $\GrpCat{\Coh}$ of $\pref{\Coh}$, of $\infty$\nobreakdash-$\Coh$\nobreakdash-groupoids, associating to a topological space $X$ its fundamental $\infty$\nobreakdash-groupoid, defined in the previous section.
\[
\xymatrixcolsep{2.5pc}
\xymatrixrowsep{1.5pc}
\xymatrix{
\Top\ar[r]\ar[rd]_{\infGrpf}
&\pref{\Coh}
\\
&\GrpCat{\Coh}\ar@{^{(}->}[u] 
}
\]
Therefore there is an induced pair of adjoint functors
\[
\GrpCat{\Coh}\Toto{1.8}{\EspClass}\Top\smsp,
\qquad
\Top\Toto{1.8}{\infGrpf}\GrpCat{\Coh}\smsp,
\]
where $\EspClass$ is the restriction of the functor $\pref{\Coh}\to\Top$ to the subcategory $\GrpCat{\Coh}$ of $\pref{\Coh}$. For a $\infty$\nobreakdash-$\Coh$\nobreakdash-groupoid $\Grp$, the topological space $\EspClass(\Grp)$ is called the \emph{classifying space} of $\Grp$.
\end{paragr}

\begin{paragr} \label{Grothconjfort}
\textbf{Grothendieck's conjecture (precise form).}
\emph{The functors $\EspClass$ and $\infGrpf$ are compatible with the weak equivalences of $\infty$\nobreakdash-groupoids and of spaces, and induce mutually quasi-inverse equivalences of the localized categories. Furthermore, different extensions $\Coh\to\Top$ of the functor $\Thetazero\to\Top$ induce isomorphic functors between the localized categories.}
\end{paragr}

\begin{paragr}
\textbf{\boldmath The $\infty$-groupoid of an object in a model category.} A similar construction can be done in any Quillen model category, such that \emph{all objects are fibrant} (\cite{PS}, section~12). For more details on this construction, see~\cite[section~4.4]{Ara}.
\end{paragr}

\section{Coherators and weak factorization systems}

\begin{paragr} \label{liftprop}
\textbf{Lifting properties.}
Fix a category $\CC$. Recall that given two arrows $i:A\to B$ and $p:X\to Y$ of $\CC$, $i$ \emph{has the left lifting property with respect to} $p$, or equivalently $p$ \emph{has the right lifting property with respect to} $i$, if for every 
commutative solid arrow square
\[
\xymatrix{
&A\ar[r]^{a}\ar[d]_{i}
&X\ar[d]^{p}
\\
&B\ar[r]_{b}\ar@{-->}[ru]^{h}
&Y
&\hskip -30pt,
}
\]
there exists a dotted arrow such that the total diagram is commutative:
\[
a=hi\hhbox{and} b=ph\smsp.
\]
If $\Classarr$ is a class of arrows of $\CC$, we denote by $\llp{\Classarr}$ (resp. $\rlp{\Classarr}$) the class of arrows of $\CC$ having the left (resp. right) lifting property with respect to all arrows in $\Classarr$. The classes $\llp{\Classarr}$ and  $\rlp{\Classarr}$ are stable under composition and retracts. Sometimes arrows in $\rlp{\Classarr}$ are called $\Classarr$\nobreakdash-\emph{fibrations} and arrows in $\cof{\Classarr}\!:\,=\llp{\rlp{\Classarr}}$, $\Classarr$\nobreakdash-\emph{cofibrations}, and if $\CC$ has a final object $*$ (resp. an initial object $\varnothing$), an object $X$ of $\CC$ is called $\Classarr$\nobreakdash-\emph{fibrant} (resp. $\Classarr$\nobreakdash-\emph{cofibrant}) if the map $X\to*$ (resp. $\varnothing\to X$) is a $\Classarr$\nobreakdash-\emph{fibration} (resp. a $\Classarr$\nobreakdash-\emph{cofibration}).
\end{paragr}

\begin{paragr} \label{wfs}
\textbf{Weak factorization systems.}
A \emph{weak factorization system} in $\CC$ is a pair $(\LClassarr,\RClassarr)$ of classes of arrows of $\CC$ such that the following conditions are satisfied:
\begin{itemize}
\item[\emph{a})] $\LClassarr$ and $\RClassarr$ are stable under retracts;
\item[\emph{b})] arrows in $\LClassarr$ have the left lifting property with respect to arrows in $\RClassarr$ (or equivalently arrows in $\RClassarr$ have the right lifting property with respect to arrows in $\LClassarr$);
\item[\emph{c})] every arrow $f$ in $\CC$ can be factored as $f=pi$, with $i\in\LClassarr$ and $p\in\RClassarr$.
\end{itemize}
It is well known that the conjunction of conditions (\emph{a}) and (\emph{b}) is equivalent to the conjunction of equalities
\[
\LClassarr=\llp{\RClassarr}\hhbox{and}\RClassarr=\rlp{\LClassarr}\smsp.
\]
\end{paragr} 

\begin{paragr} \label{celmaps}
\textbf{Cellular maps.}
Let $\CC$ be a cocomplete category and $\Classarr$ a class of arrows in~$\CC$. A $\Classarr$\nobreakdash-\emph{cellular} map in $\CC$ is a map obtained as transfinite composition of pushouts of (small) sums of arrows in $\Classarr$. The class of $\Classarr$\nobreakdash-cellular maps in $\CC$ is denoted by $\cell{\Classarr}$. It is easy to verify that $\cell{\Classarr}$ is stable under pushouts, sums and transfinite composition, and that it contains isomorphisms. It is the smallest class containing $\Classarr$ and stable under pushouts and transfinite composition. It is equal to the class of maps obtained by transfinite composition of pushouts of arrows in~$\Classarr$. There is an inclusion 
\[
\cell{\Classarr}\subset\cof{\Classarr}=\llp{\rlp{\Classarr}}\smsp.
\]
An object $X$ of $\CC$ is called $\Classarr$\nobreakdash-\emph{cellular} if the map from the initial object of $\CC$ to $X$ is a $\Classarr$\nobreakdash-cellular map.
\end{paragr}

\begin{paragr} \label{smallobjarg}
\textbf{Set of arrows admissible for the small object argument.}
Let $\CC$ be a cocomplete category and $\Setarr$ a (small) set of arrows in $\CC$. The set of arrows $\Setarr$ is said to be \emph{admissible for the small object argument}, or more simply \emph{admissible}, if there exists a well ordered set $\wellord$ satisfying the following conditions:
\begin{itemize}
\item[\emph{a})] $\wellord$ does \emph{not} have a maximal element;
\item[\emph{b})] for every $\wellord$\nobreakdash-diagram $(X_j)_{j\in\wellord}$ such that all maps $X_j\to X_{j'}$, $j\leqslant j'$, are $\Setarr$\nobreakdash-cellular maps, and every arrow $A\to B$ in $\Setarr$, the canonical map
\[
\varinjlim_j\Hom(A,X_j)\toto\Hom(A,\varinjlim_j X_j)
\]
is a bijection.
\end{itemize}
For example, if all domains of maps in $\Setarr$ are $\alpha$\nobreakdash-\emph{presentable} for some regular cardinal $\alpha$, then $\Setarr$ is admissible. In particular if $\CC$ is \emph{locally presentable}, then \emph{any} (small) set of arrows in $\CC$ is admissible. Recall the small object argument:
\end{paragr}

\begin{prop} \label{smallobjargprop}
Let $\CC$ be a cocomplete category and $\Setarr$ a (small) set of arrows in $\CC$. If $\Setarr$ is admissible for the small object argument, then every arrow $f$ of $\CC$ can be factored as $f=pi$, with $i\in\cell{\Setarr}$ and $p\in\rlp{\Setarr}$. In particular,
\[
(\cof{\Setarr},\rlp{\Setarr})\qquad(\cof{\Setarr}=\llp{\rlp{\Setarr}})
\] 
is a weak factorization system. Moreover, $\cof{\Setarr}$ is the smallest class of arrows in $\CC$ containing $\Setarr$ stable under pushouts, transfinite composition and retracts, and it is equal to the class of retracts of arrows in $\cell{\Setarr}$.
\end{prop}

\begin{paragr} \textbf{Finitely presentable objects.}
Recall that an object $A$ of $\CC$ is called \emph{finitely presentable} if it is $\aleph_0$\nobreakdash-presentable, \ie~if for every small \emph{filtered} category $\filt$, and every functor $X:\filt\to\CC$, the canonical map
\[
\varinjlim_\filt\Hom(A,X)\toto\Hom(A,\varinjlim_\filt X)
\]
is a bijection. When the domains of the arrows in a (small) set $\Setarr$ are finitely presentable, $\Setarr$ is in particular admissible, and the following lemma gives a more precise description of $\cell{\Setarr}$: 
\end{paragr}

\begin{lemme} \label{lemmeclef}
Let $\CC$ be a cocomplete category and $\Setarr$ a (small) set of arrows in $\CC$. If the domains of the arrows in $\Setarr$ are finitely presentable, then an arrow $X\to Y$ is in $\cell{\Setarr}$ if and only if there exists a sequence of maps
\[
\xymatrixcolsep{1.8pc}
\xymatrix{
X=X_0\ar[r]^-{i_0}
&X_1\ar[r]^-{i_1}
&\cdots\ar[r]^-{i_{n-1}}
&X_n\ar[r]^-{i_n}
&X_{n+1}\ar[r]^-{i_{n+1}}
&\cdots\ar[r]
&Y\simeq\varinjlim X_n
}\smsp,
\]
identifying $X\to Y$ to the canonical map $X=X_0\to\varinjlim X_n\simeq Y$, such that for every $n\geqslant0$, the arrow $i_n$ is a pushout of a (small) sum of arrows in $\Setarr$.
\end{lemme}

For a proof of this lemma, see appendix~\ref{proofclef}, proposition~\ref{propclef}.

\begin{paragr} \label{remlemmeclef}
Let $\Classarr$ be a class of arrows in $\CC$ and $\lambda$ an ordinal. We denote by $\cell[\lambda]{\Classarr}$ the class of arrows obtained by $\lambda$\nobreakdash-indexed transfinite composition of pushouts of (small) sums of arrows in $\Classarr$. With this notation, the lemma~\ref{lemmeclef} says that if $\Setarr$ is a (small) set of arrows in $\CC$, with finitely presentable domains, then $\cell{\Setarr}=\cell[\omega]{\Setarr}$, where $\omega$ is the smallest countable ordinal. More generally, it can be shown that if $\Setarr$ is a (small) set of arrows in $\CC$, with $\alpha$\nobreakdash-presentable domains, for some regular cardinal $\alpha$, then $\cell{\Setarr}=\cell[\lambda]{\Setarr}$, where $\lambda$ is the smallest ordinal of cardinality $\alpha$.
\end{paragr}

\begin{paragr} \textbf{Globular extensions under $\Thetazero$.}
Let $\mc{\Thetazero}{\Cat}$ be the category of small categories under $\Thetazero$, whose objects are pairs $(\C,\Thetazero\to\C)$, where $\C$ is a small category and $\Thetazero\to\C$ a functor, and whose morphisms are (strictly) commutative triangles
\[
\xymatrixcolsep{2.8pc}
\xymatrixrowsep{.7pc}
\xymatrix{
&&\C\ar[dd]
\\
&\Thetazero\ar[ru]\ar[rd]
\\
&&\C'
&\hskip -40pt.
}\hskip 40pt
\]
Denote by $\CatExtGl$ the full subcategory of $\mc{\Thetazero}{\Cat}$ whose objects are pairs $(\C,\Thetazero\to\C)$ such that $\C$, endowed with the composite functor
\[
\xymatrix{
\Glob\ar[r]
&\Thetazero\ar[r]
&\C\quad,
}
\]
is a globular extension and the functor $\Thetazero\to\C$ a morphism of globular extensions (\ie~preserves globular sums). It is easy to see that any morphism of $\CatExtGl$ preserves globular sums and therefore defines a morphism of the underlying globular extensions. In other terms, the category $\CatExtGl$ is the category of globular extensions \emph{under} the universal globular extension $\Thetazero$. It's nothing else then the category of globular extensions endowed with a fixed choice of objects representing the globular sums, and morphisms of globular extensions compatible with these choices.
\end{paragr}

\begin{paragr} \textbf{The left adjoint to the inclusion functor $\CatExtGl\to\mc{\Thetazero}{\Cat}$.}
Let $(\C,\Thetazero\to\C)$ be any object of $\mc{\Thetazero}{\Cat}$. By a standard categorical construction~\cite[section~2.6]{Ara}, there exists a globular extension $\GlobUniv{\C}$ with a morphism $\C\to\GlobUniv{\C}$ in $\mc{\Glob}{\Cat}$
\[
\xymatrix{
&\Glob\ar[r]\ar[rrd]
&\Thetazero\ar[r]
&\C\ar[d]
\\
&&&\GlobUniv{\C}
&\hskip -30pt,
}
\]
satisfying the following universal property: for every morphism $\C\to\C'$ of globular extensions under $\Thetazero$, there exists a unique morphism of globular extensions $\GlobUniv{\C}\to\C'$ such that the following triangle is commutative
\[
\xymatrix{
&\C\ar[d]\ar[rd]
\\
&\GlobUniv{\C}\ar[r]
&\C'
&\hskip -30pt.
}
\]
Moreover, the functor $\C\to\GlobUniv{\C}$ induces a bijection on the sets of objects. It is easy to see that the assignment
\[
(\C,\Thetazero\to\C)\longmapsto(\GlobUniv{\C},\Thetazero\to\C\to\GlobUniv{\C})
\]
defines a functor $\mc{\Thetazero}{\Cat}\to\CatExtGl$, which is a left adjoint of the full and faithful inclusion $\CatExtGl\to\mc{\Thetazero}{\Cat}$, the canonical map $\C\to\GlobUniv{\C}$ corresponding to the unit of the adjunction.
In particular, as the category $\mc{\Thetazero}{\Cat}$ is complete and cocomplete, the same holds for $\CatExtGl$. Furthermore, one can easily verify that the inclusion $\CatExtGl\to\mc{\Thetazero}{\Cat}$ preserves filtered colimits, and as a consequence the category $\CatExtGl$ is locally presentable.
\end{paragr}

\begin{paragr} \textbf{The generating cofibrations.}
Let $\Tree$ be an object of $\Thetazero$ and $i$ an integer, $i\geqslant0$. Denote by $\SourceI{\Tree}{i}$ the globular extension obtained from $\Thetazero$ by formally adding two arrows
\[
\xymatrixcolsep{3.5pc}
\xymatrix{
\Dn{i} \ar@<1ex>[r]^-{\mathUniv{f}=\mathUniv{f}_{(\Tree,i)}} \ar@<-1ex>[r]_-{\mathUniv{g}=\mathUniv{g}_{(\Tree,i)}}
&\Tree
}
\]
satisfying (in case $i>0$) the relations
\[
\mathUniv{f}\Ths{i}=\mathUniv{g}\Ths{i}\smsp,\quad \mathUniv{f}\Tht{i}=\mathUniv{g}\Tht{i}\smsp,
\]
\ie~by formally adding a pair of parallel arrows, and $\ButI{\Tree}{i}$ the globular extension obtained from $\SourceI{\Tree}{i}$  by formal adjunction of a lifting $\mathUniv{h}=\mathUniv{h}_{(\Tree,i)}$ for the pair $(\mathUniv{f}_{(\Tree,i)},\mathUniv{g}_{(\Tree,i)})$. Define
\[
\SetArr=\{\SourceI{\Tree}{i}\toto\ButI{\Tree}{i}\mid i\geqslant0\,,\ \Tree\in\Ob\,\Thetazero\}\smsp.
\]
Observe that as $\Thetazero$ is a small category, $\SetArr$ is a (small) set. On the other hand, we have the following immediate lemma:
\end{paragr}

\begin{lemme} 
The categories $\SourceI{\Tree}{i}$ \emph{(}as well as $\ButI{\Tree}{i}$\emph{),} considered as categories under $\Thetazero$, are finitely presentable objects of $\CatExtGl$.
\end{lemme}

In particular, the set $\SetArr$ is admissible for the small object argument, and applying proposition~\ref{smallobjargprop}, we get:

\begin{prop} \label{factsystcohprop}
Every arrow $f$ of $\CatExtGl$ can be factored as $f=pi$, with $i\in\cell{\SetArr}$ and $p\in\rlp{\SetArr}$. In particular, 
\[
(\cof{\SetArr},\rlp{\SetArr})\qquad(\cof{\SetArr}=\llp{\rlp{\SetArr}})
\] 
is a weak factorization system. Moreover, $\cof{\SetArr}$ is the smallest class of arrows in $\CatExtGl$ containing $\Setarr$ stable under pushouts, transfinite composition and retracts, and it is equal to the class of retracts of arrows in $\cell{\Setarr}$.
\end{prop}

We can now state the principal result of this section, roughly saying that \Gr\nobreakdash-coher\-ators are exactly the $\SetArr$\nobreakdash-fibrant $\SetArr$\nobreakdash-cellular objects of $\CatExtGl$. More precisely:

\begin{thm} \label{thmclef}
An object $(\C,\Thetazero\to\C)$ of $\CatExtGl$ is $\SetArr$\nobreakdash-fibrant and $\SetArr$\nobreakdash-cellular if and only if $\C$, endowed with the composite map $\Glob\to\Thetazero\to\C$, is a \Gr\nobreakdash-coherator. 
\end{thm}

\begin{proof}
Let $(\C,\Thetazero\to\C)$ be an object of  $\CatExtGl$. Observe that, for $\Tree$ an object of $\Thetazero$ and $i$ an integer, $i\geqslant0$, by definition of $\SourceI{\Tree}{i}$, there is a bijection of the set $\Hom_{\CatExtGl}(\SourceI{\Tree}{i},\C)$ with the set of pairs of parallel arrows in $\C$, with domain $\Dn{i}$ and codomain $\Tree$. Furthermore, a map from $\SourceI{\Tree}{i}$ to $\C$ has an extension to $\ButI{\Tree}{i}$ if and only if the corresponding pair of parallel arrows has a lifting.
\smallbreak

On the other hand, it is easy to verify that if $(f_k,g_k)_{k\in K}$,  $f_k,g_k:\Dn{i_k}\to\Tree_k$, is a family of pairs of parallel arrows in $\C$, the universal globular extension $\C\to\C'$ obtained from $\C$ by formally adding a lifting for every pair in this family, can be constructed by the following pushout square in $\CatExtGl$
\[
\xymatrix{
&\mathop{\amalg}\limits_{k\in K}\SourceI{\Tree_k}{i_k} \ar@<.7ex>[r]\ar[d]
&\C\vrule depth 6pt width 0pt\ar[d]
\\
&\mathop{\amalg}\limits_{k\in K}\ButI{\Tree_k}{i_k}\ar@<.7ex>[r]
&\C'\vrule depth 6pt width 0pt
&\hskip -30pt,
}
\]
where the upper horizontal arrow is defined by the maps $\SourceI{\Tree_k}{i_k}\to\C$ corresponding to the pairs of parallel arrows $(f_k,g_k)$, $k\in K$.
\smallbreak

Suppose now that $\C$ is a \Gr\nobreakdash-coherator. Then the functor $\Thetazero\to\C$ induces a bijection of the sets of objects, and the previous considerations imply that the condition (\emph{a}) of the definition of a \Gr\nobreakdash-coherator means exactly that $(\C,\Thetazero\to\C)$ is a $\SetArr$\nobreakdash-fibrant object of $\CatExtGl$, and that the condition (\emph{b}) of the definition implies that $(\C,\Thetazero\to\C)$ is $\Setarr$\nobreakdash-cellular. The converse is an immediate consequence of lemma~\ref{lemmeclef}.
\end{proof}

\begin{paragr} \label{pseudocoh} 
\textbf{\Gr-Pseudo-coherators.}
A \Gr-\emph{pseudo-coherator} is a globular extension\break $(\C,\Glob\to\C)$ such that every pair of parallel arrows in $\C$ \emph{with codomain a globular sum} has a lifting in $\C$. As observed in the proof above, the condition (\emph{a}) in the definition of a \Gr\nobreakdash-coherator can be replaced by the condition of being a \Gr\nobreakdash-pseudo-coherator. If $(\C,\Glob\to\C)$ is a globular extension with $\C$ a small category, and $\Thetazero\to\C$ the unique (up to unique isomorphism) morphism of globular extensions defined by the universal property of $\Thetazero$, then $(\C,\Thetazero\to\C)$ is an object of $\CatExtGl$, and $(\C,\Glob\to\C)$ is a \Gr\nobreakdash-pseudo-coherator if an only if $(\C,\Thetazero\to\C)$ is a $\SetArr$\nobreakdash-fibrant object of $\CatExtGl$. 
\end{paragr}

\begin{prop} \label{proppseudo}
Let $\Coh$ be a \Gr\nobreakdash-coherator and $\C$ a \Gr\nobreakdash-pseudo-coherator. Then there exists a morphism of globular extensions from $\Coh$ to $\C$.
\end{prop}

\begin{proof}
Choosing a small full subcategory of $\C$ containing globular sums, we can suppose that $\C$ is small. By the universal property of $\Thetazero$, $\Coh$ and $\C$ define objects of $\CatExtGl$, denoted also by $\Coh$ and $\C$, the first being, in particular, a $\Setarr$\nobreakdash-cellular object (\cf~\ref{thmclef}) and the second being $\Setarr$\nobreakdash-fibrant (\cf~\ref{pseudocoh}). Therefore the proposition is a consequence of the lifting properties between $\Setarr$\nobreakdash-cellular maps and $\Setarr$\nobreakdash-fibrations.
\[
\xymatrix{
\varnothing\ar[d]\ar[r]
&\C\ar[d]
\\
\Coh\ar[r]\ar@{-->}[ru]
&\ast
}
\]
\vskip -17pt
\end{proof}

\begin{paragr} \textbf{The \Gr-pseudo-coherator $\Top$.}
The considerations of paragraph~\ref{fondinfgrpd} show that the category $\Top$ of topological spaces, endowed with the canonical functor $\Glob\to\Top$, defined by topological disks (\cf~\ref{topglob}), is a \Gr\nobreakdash-pseudo-coherator. For $\Coh$ a \Gr\nobreakdash-coherator, proposition~\ref{proppseudo} gives a formal interpretation of the construction of the topological $\infty$\nobreakdash-$\Coh$\nobreakdash-cogroupoid $\Coh\to\Top$, explained in paragraph~\ref{fondinfgrpd}.
\end{paragr}

\section{A definition of lax $\infty$-categories}

\begin{paragr}
\textbf{Globular theories.}
A \emph{globular theory} is a globular extension $(\C,\Glob\to\C)$ such that the morphism of globular extensions $\Thetazero\to\C$, defined by the universal property of $\Thetazero$, is faithful and induces a bijection of the sets of isomorphism classes of objects. Replacing $\C$ by an equivalent category, we can suppose, without loss of generality, that this functor induces a bijection of the sets of objects, as we will always assume in the sequel. We will identify $\Thetazero$ to a (non-full) subcategory of $\C$, and we will say that an arrow $f$ of $\C$ is \emph{globular} if $f$ is in $\Thetazero$. The arrow $f$ will be called \emph{algebraic} if for every decomposition $f=gf'$, with $g$ globular, $g$ is an identity. A morphism of globular theories is a morphism of the underlying globular extensions.
\end{paragr}

\begin{paragr}
\textbf{Generalized cosource and cotarget maps.} Let 
\[
\Tree=\Dn{i_1}\amalg^{}_{\Dn{i'_1}}\Dn{i_2}\amalg^{}_{\Dn{i'_2}}\cdots\,\,\amalg^{}_{\Dn{i'_{m-1}}}\!\!\Dn{i_m}\smsp
\]
an object of $\Thetazero$. The \emph{dimension} of $\Tree$ is the integer $i=\max\{i_k\,|\,1\leqslant k\leqslant m\}$. If $i>0$, we define an object $\Thbord\Tree$ of $\Thetazero$, of dimension $i-1$,
\[
\Thbord\Tree:\,=\Dn{\Thbordind{i_1}}\amalg^{}_{\Dn{i'_1}}\Dn{\Thbordind{i_2}}\amalg^{}_{\Dn{i'_2}}\cdots\,\,\amalg^{}_{\Dn{i'_{m-1}}}\!\!\Dn{\Thbordind{i_m\kern -2pt}}\smsp\,,
\]
where, for $1\leqslant k\leqslant m$,
\[
\Thbordind{i_k}=\left\{
\begin{aligned}
&i_k\smsp,\quad\hskip 12pt i_k<i\smsp,\cr
\noalign{\vskip 3pt}
&i-1\smsp,\quad i_k=i\smsp,
\end{aligned}
\right.
\]
and arrows $\Ths{\Tree},\Tht{\Tree}:\Thbord\Tree\to\Tree$ of $\Thetazero$
\[
\Ths{\Tree}:\,=\Thbords{1}\amalg^{}_{\Dn{i'_1}}\!\Thbords{2}\amalg^{}_{\Dn{i'_2}}\cdots\,\,\amalg^{}_{\Dn{i'_{m-1}}}\!\!\!\Thbords{m}\smsp\,,\quad
\Tht{\Tree}:\,=\Thbordt{1}\amalg^{}_{\Dn{i'_1}}\!\Thbordt{2}\amalg^{}_{\Dn{i'_2}}\cdots\,\,\amalg^{}_{\Dn{i'_{m-1}}}\!\!\!\Thbordt{m}\smsp\,,
\]
where, for $1\leqslant k\leqslant m$,
\[
\Thbords{k}=\left\{
\begin{aligned}
&\id{\Dn{i_k}}\smsp,\quad i_k<i\smsp,\cr
\noalign{\vskip 3pt}
&\Ths{i_k}\smsp,\quad\hskip 5pt i_k=i\smsp,
\end{aligned}
\right.
\qquad
\Thbordt{k}=\left\{
\begin{aligned}
&\id{\Dn{i_k}}\smsp,\quad i_k<i\smsp,\cr
\noalign{\vskip 3pt}
&\Tht{i_k}\smsp,\quad\hskip 5pt i_k=i\smsp.
\end{aligned}
\right.
\]
\end{paragr}

\begin{paragr}
\textbf{Admissible pairs of arrows.} Let $\C$ be a globular theory. A \emph{pair of parallel arrows admissible for a theory of} $\infty$\nobreakdash-\emph{categories}, or more simply an \emph{admissible pair} of arrows of $\C$, is a pair $(f,g)$ of parallel arrows $f,g:\Dn{i}\to\Tree$ (\cf~\ref{cohgr}) such that either the arrows $f$ and $g$ are algebraic, or there exist decompositions
\[
f=\Ths{\Tree}f'\hhbox{and}g=\Tht{\Tree}g'
\]
\[
\vrule height 18pt depth 50pt width 0pt
\xymatrix{
\Dn{i}\ar[rr]^f\ar[rd]_{f'}
&&\Tree
&&\Dn{i}\ar[rr]^g\ar[rd]_{g'}
&&\Tree
\\
&\Thbord\Tree \ar[ru]_{\Ths{\Tree}}
&&&&\Thbord\Tree\ar[ru]_{\Tht{\Tree}}
}
\]
with $f'$ and $g'$ algebraic. The property for a parallel pair of arrows $(f,g)$ to be admissible is not symmetric: if $(f,g)$ is admissible, $(g,f)$ \emph{need not} be admissible.
\end{paragr}

\begin{paragr}
\textbf{\boldmath Coherators for a theory of $\infty$-categories.} \label{cohcat}
A \emph{coherator for a theory of $\infty$\nobreakdash-categories}, or more simply a \emph{\Ctg-coherator}, is a globular theory $\Glob\to\CohCat$ satisfying the following two conditions:
\smallbreak

\emph{a}) Every admissible pair of arrows in $\CohCat$ has a lifting in $\CohCat$.
\smallskip

\emph{b}) There exists a ``tower'' of globular theories (called \emph{tower of definition} of the \Ctg\nobreakdash-coherator $\CohCat$) with colimit $\CohCat$
\[
\xymatrixcolsep{1.5pc}
\xymatrix{
\Glob \ar[r]
&\CohCat_0\ar[r]
&\CohCat_1\ar[r]
&\cdots\ar[r]
&\CohCat_n\ar[r]
&\CohCat_{n+1}\ar[r]
&\cdots\ar[r]
&\CohCat\simeq\varinjlim\CohCat_n
}\smsp,
\]
where for every $n\geqslant0$, $\CohCat_n\to\CohCat_{n+1}$ is a morphism of globular theories, satisfying the following properties:
\begin{itemize}
\item[$b_0$\hskip 1pt)] $\Glob\to\CohCat_0$ is a globular completion;
\item[$b_n$)] for every $n\geqslant0$, there exists a family of admissible pairs of arrows in $\CohCat_n$ such that $\CohCat_{n+1}$ is the universal globular extension obtained from $\CohCat_n$ by formally adding a lifting for every pair in this family.
\end{itemize}
Condition $(b_0)$ implies that the canonical functor from $\Thetazero$ to $\CohCat_0$ is an isomorphism. We will usually suppose that $\CohCat_0$ is \emph{equal} to $\Thetazero$. Condition $(b_n)$ means, more precisely, that there exists a family $(\mathUniv{f}_i,\mathUniv{g}_i)_{i\in I_n}$  of admissible pairs of arrows in $\CohCat_n$, and for every $i\in I_n$, a lifting $\mathUniv{h}_i$ in $\CohCat_{n+1}$ of the image of the pair $(\mathUniv{f}_i,\mathUniv{g}_i)$ in $\CohCat_{n+1}$, satisfying the following universal property. For every globular extension $\Glob\to\C$ and every morphism of globular extensions $\CohCat_n\to\C$, if for every $i\in I_n$ a lifting $h_i$ of the image of the pair $(\mathUniv{f}_i,\mathUniv{g}_i)$ in $\C$ is given, then there exists a unique morphism of globular extensions $F:\CohCat_{n+1}\to\C$ such that for every $i$ in $I_n$, $F(\mathUniv{h}_i)=h_i$ and such that the triangle
\[
\xymatrixrowsep{1.7pc}
\xymatrix{
\CohCat_n\ar[r]\ar[rd]
&\CohCat_{n+1}\ar[d]^{F}
\\
&\C
}
\]
is commutative. 
\end{paragr}

\begin{paragr} \textbf{\boldmath Lax $\infty$-categories.} \label{definfcat}
Let $\Glob\to\CohCat$ be a \Ctg\nobreakdash-coherator. A \emph{$\infty$\nobreakdash-category of type $\CohCat$}, or more simply a \emph{$\infty$\nobreakdash-$\CohCat$\nobreakdash-category}, is a presheaf $\Catg:\CohCat^\op\to\Ens$ on $\CohCat$ such that the functor $\Catg^\op:\CohCat\to\Ens^\op$ preserves globular sums. In other terms, for any globular sum in $\CohCat$, the canonical map
\[
\Catg(\Dn{i_1}\amalg^{}_{\Dn{i'_1}}\cdots\,\,\amalg^{}_{\Dn{i'_{m-1}}}\Dn{i_m})\toto
\Catg(\Dn{i_1})\times^{}_{\Catg(\Dn{i'_1})}\cdots\,\,\times^{}_{\Catg(\Dn{i'_{m-1}})}\Catg(\Dn{i_m})
\smsp
\]
is a bijection, the right hand side being the globular product, limit of the diagram
\[
\xymatrixrowsep{.2pc}
\xymatrixcolsep{1pc}
\xymatrix{
\Catg(\Dn{i_1})\ar[rdd]_(.37){\Catg(\Ths[i'_1]{i_1})}
&&\Catg(\Dn{i_2})\ar[ldd]^(.37){\!\Catg(\Tht[i'_1]{i_2})}\ar[rdd]^(.65){\!\!\!\!\Catg(\Ths[i'_2]{i_2})}
&&&&\Catg(\Dn{i_m})\ar[ldd]^(.37){\!\Catg(\Tht[i'_{m-1}]{i_m})}
\\
&&&&\cdots
\\
&\Catg(\Dn{i'_1})
&&\Catg(\Dn{i'_2})
&&\Catg(\Dn{i'_{m-1}})
&\hskip 38pt.
}
\] 
The \emph{category of $\infty$\nobreakdash-$\CohCat$\nobreakdash-categories} is the full subcategory of $\pref{\CohCat}$, category of presheaves on $\CohCat$, whose objects are $\infty$\nobreakdash-$\CohCat$\nobreakdash-categories.
\end{paragr}

\begin{paragr} \textbf{Examples of \Ctg-coherators.} \label{excohcat}
As in the case of \Gr\nobreakdash-coherators, there is a general method for constructing inductively \Ctg\nobreakdash-coherators. Take $\CohCat_0=\Thetazero$. Suppose that $\CohCat_n$ is defined and choose a set $\Adm_n$ of admissible pairs of arrows in $\CohCat_n$. Define $\CohCat_{n+1}$ as the universal globular extension obtained by formally adding a lifting for each pair in $\Adm_n$, and $\CohCat$ as the colimit $\CohCat=\varinjlim\CohCat_n$. In order for $\CohCat$ to be a \Ctg\nobreakdash-coherator, one has to verify first that the globular extensions $\CohCat_n$, $n\geqslant0$, are globular theories, and secondly that $\CohCat$ satisfies the condition (\emph{a}) in~\ref{cohcat}. The first of these conditions is conjectured to be always true, and is proved in~\cite[section~5.4]{Ara} with a mild hypothesis always verified in the examples: it is enough that $\Adm_0$ contains all the pairs $(\id{\Dn{i}},\id{\Dn{i}})$, $i\geqslant0$.
For the second condition to be verified, it is sufficient  (but not necessary) that every admissible pair of arrows in $\CohCat$ is the image of a pair in $\Adm_n$, for some $n\geqslant0$. Three important examples can be constructed (among many others) by this method.
\smallbreak

1) \textbf{The canonical \Ctg-coherator.} This example is obtained by taking $\Adm_n$ to be the set of \emph{all} admissible pairs of arrows in $\CohCat_n$. 
\smallbreak

2) \textbf{The Batanin-Leinster \Ctg-coherator.} It is obtained by defining $\Adm_n$ to be the set of admissible pairs of arrows in $\CohCat_n$ that are not the image of a pair in $\Adm_{n'}$, for some $n'<n$~\cite[4.1.4]{Ara}.

3) \textbf{The canonical reduced \Ctg-coherator.} It is constructed by taking $\Adm_n$ to be the set of admissible pairs of arrows in $\CohCat_n$ that do not have already a lifting in~$\CohCat_n$.
\smallbreak

It is easily seen that examples 1 and 2 satisfy the sufficient condition stated above. The example 3 does \emph{not} satisfy this condition; nevertheless, it is clear that it defines a \Ctg\nobreakdash-coherator. It is possible to put even more restrictive conditions on the sets $\Adm_n$ and still obtain a \Ctg\nobreakdash-coherator. It seems that it is not possible to find a minimal way for choosing the sets $\Adm_n$.
\end{paragr}

\begin{paragr}
\textbf{Structural maps.}
Let $\CohCat$ be a \Ctg-coherator. Those of the structural maps for $\infty$\nobreakdash-groupoids defined in paragraph~\ref{structmor} that do not concern inverses (\ie~\ref{bincompl1}\nobreakdash-\ref{trianglecons}) exist equally well in $\infty$\nobreakdash-$\CohCat$\nobreakdash-categories. To see this, one has to verify that the pairs of parallel arrows giving rise to these structural maps are admissible. This is proved in~\cite[4.2.7]{Ara}, under the mild hypothesis considered above. It is conjectured that it is always true.
\end{paragr}

\appendix
\section{Proof of lemma~\ref{lemmeclef}\label{proofclef}}

\begin{paragr}
Let $\CC$ be a cocomplete category and $\Classarr$ a class of arrows of $\CC$. Denote $\imds{\Classarr}$ the class of pushouts of (small) sums of arrows in $\Classarr$. It is easy to verify that $\imds{\Classarr}$ is stable under pushouts and sums, and that it contains isomorphisms. It is the smallest class containing $\Classarr$ and stable under pushouts and sums.
There is an inclusion 
\[
\imds{\Classarr}\subset\cell{\Classarr}\smsp.
\]
\end{paragr}

\begin{paragr}
A commutative square 
\[
\Square=\hskip 5pt
\raise 20pt
\vbox{
\xymatrix{
X\ar[r]^{f}\ar[d]_{x}
&Y\ar[d]^{y}
\\
X'\ar[r]_{f'}
&Y'
}
}\hskip 25pt
\]
in $\CC$ is called $\Classarr$\nobreakdash-special if there exists a commutative square
\[
\xymatrix{
&A\ar[r]^{i}\ar[d]_{a}
&B\ar[d]^{b}
\\
&X'\ar[r]_{f'}
&Y'
&\hskip -30pt,
}
\]
with $i$ in $\imds{\Classarr}$, such that the square
\[
\xymatrix{
A\amalg X\ar[r]^{i\amalg f}\ar[d]_{(a,x)}
&B\amalg Y\ar[d]^{(b,y)}
\\
X'\ar[r]_{f'}
&Y'
}
\]
is cocartesian (a pushout square). The property of being a $\Classarr$\nobreakdash-special square is a property of a square in an \emph{oriented} plane: if $\Square$ is $\Classarr$\nobreakdash-special the square
\[
\xymatrix{
X\ar[r]^{x}\ar[d]_{f}
&X'\ar[d]^{f'}
\\
Y\ar[r]_{y}
&Y'
}
\]
need not be $\Classarr$\nobreakdash-special.
\end{paragr}

\begin{lemme}
Let $\CC$ be a category, consider a commutative diagram in $\CC$
\[
\xymatrix{
&A\ar[r]^{i}\ar[d]_{u} \ar@<2.3ex>@{}[d]^(.45){\textstyle(1)}
&B\ar[d]^{v}
\\
A'\ar[r]^{u'}\ar[d]_{i'}\ar@<2.2ex>@{}[d]^(.49){\textstyle(1')}
&X\ar[r]^{g}\ar[d]_{g'\kern -1.5pt}\ar@<2.3ex>@{}[d]^(.49){\textstyle(2)}
&Y\ar[d]^{h}
\\
B'\ar[r]_{v'}
&Y'\ar[r]_{h'}
&Z
}
\]
and the induced commutative square
\[
\xymatrixcolsep{2.3pc}
\xymatrix{
&A\amalg A'\ar[r]^{i\,\amalg\,i'}\ar[d]_{(u,u')}\ar@<5.ex>@{}[d]^(.51){\textstyle(3)}
&B\amalg B'\ar[d]^{(hv,h'v')}
\\
&X\ar[r]_{hg{\scriptscriptstyle\,=\,}h'g'}
&Z
&\hskip -30pt.
}
\]
If $(1)$ and $(1')$ are cocartesian (pushout) squares, then the square~$(3)$ is cocartesian if and only if the square~$(2)$ is cocartesian.
\end{lemme}

\begin{proof}
Consider the following commutative diagram
\[
\xymatrixcolsep{1.pc}
\xymatrixrowsep{2.6pc}
\xymatrix{
&A \ar[rrrr]^{i}\ar[d]^(.42){\textstyle\hskip 43pt (\mathrm{I})}
&&&&B\ar[rd]
&&A'\ar[rrrr]^{i'}\ar[ld]
&&&&B'\ar[d]^(.42){\textstyle\hskip -65pt (\mathrm{I}')}
\\
&A\amalg A'\ar[d]_{(u,u')}^(.5){\textstyle\hskip 45pt (\mathrm{II})}\ar[rrrrr]^{i\,\amalg\,\id{\kern -1ptA'}}
&&&&&B\amalg A'\ar[d]^{(v,gu')}\ar[rrrrr]^{\id{B}\kern 1pt\amalg\,i'}
&&&&&B\amalg B'\ar[d]^{(hv,h'v')}\ar[d]^(.5){\textstyle\hskip -65pt (\mathrm{II}')}
\\
&X\ar[rrrrr]_{g}
&&&&&Y\ar[rrrrr]_{h}
&&&&&Z
&,
}
\]
where the squares $(\mathrm{I})$ and $(\mathrm{I}')$ are cocartesian. Suppose that the squares $(1)$ and $(1')$ are cocartesian. As the squares $(\mathrm{I})$ and $(\mathrm{II})\circ(\mathrm{I})=(1)$ are cocartesian, so is $(\mathrm{II})$. Therefore the square $(3)=(\mathrm{II}')\circ(\mathrm{II})$ is cocartesian if and only if $(\mathrm{II}')$ is cocartesian. On the other hand as $(\mathrm{I}')$ is cocartesian, $(\mathrm{II}')$ is cocartesian if and only if $(\mathrm{II}')\circ(\mathrm{I}')=(2)\circ(1')$ is cocartesian. Finally, as $(1')$ is cocartesian, $(2)\circ(1')$ is cocartesian if and only if $(2)$ is cocartesian, which proves the lemma.
\end{proof}

\begin{lemme} \label{carspecial}
Let $\CC$ be a cocomplete category and $\Classarr$ a class of arrows of $\CC$. A commutative square in $\CC$
\[
\Square=\hskip 5pt
\raise 20pt
\vbox{
\xymatrix{
X\ar[r]^{f}\ar[d]_{x}
&Y\ar[d]^{y}
\\
X'\ar[r]_{f'}
&Y'
}
}\hskip 25pt
\]
is $\Classarr$-special if and only if there exists a commutative diagram
\[
\xymatrix{
&A\ar[r]^{i}\ar[d]_{a}
&B\ar[d]^{b}
\\
&X'\ar[r]_{f'}
&Y'
&\hskip -30pt,
}
\]
with $i$ in $\imds{\Classarr}$, such that the induced square
\[
\xymatrix{
A\ar[r]^{i}\ar[d]
&B\ar[d]^{b}
\\
X'\amalg_XY\ar[r]
&Y'
}
\]
\emph{(}where $A\to X'\amalg_XY$ is the composite map $A\To{a}X'\To{\mathrm{can}}X'\amalg_XY$ and $X'\amalg_XY\to Y'$ is the canonical map defined by the commutative square $\Square$\emph{)} is cocartesian.
\end{lemme}

\begin{proof}
This lemma is an immediate consequence of the previous one, applied to the following solid commutative diagram
\[
\xymatrix{
&&X\ar[r]^{f}\ar[d]_{x}
&Y\ar[d]\ar@/^2pc/@{..>}[dd]^{y}
\\
&A\ar[r]^{a}\ar[d]_{i} 
&X'\ar[r]\ar[d]\ar@{..>}[rd]^{f'}
&X'\amalg_XY\ar[d]
\\
&B\ar[r]\ar@/_1.5pc/@{..>}[rr]_{b}
&B\amalg_AX'\ar[r]
&Y'
&\hskip -30pt,
}
\]
defined in the obvious way.
\end{proof}

\begin{lemme} \label{proprspecial}
Let $\CC$ be a cocomplete category and $\Classarr$ a class of arrows of $\CC$. 
\smallbreak

\emph{a)} Every cocartesian square in $\CC$ is $\Classarr$-special.
\smallbreak

\emph{b)} Let
\[
\xymatrix{
X\ar[r]^{f}\ar[d]_{x}
&Y\ar[d]^{y}
\\
X'\ar[r]_{f'}
&Y'
}
\]
be a $\Classarr$-special square. If $f$ is in $\imds{\Classarr}$, then $f'$ is in $\imds{\Classarr}$, too. 
\smallbreak

\emph{c)} The class of $\Classarr$-special squares is stable under ``vertical'' transfinite composition. More precisely, let $\wellord$ be a well-ordered set, with minimal element $\minwellord$, and
\[
X:\wellord\toto\Ar{\CC}\smsp,\quad\wo\longmapsto X_j=\ \,X_{\wo,0}\to X_{\wo,1}
\]
a functor from $\wellord$ to the category of arrows of $\CC$, satisfying the following conditions:
\begin{itemize}
\item[{\rm i)}] if $\Succ{\wo}$ is the successor of $\wo$ in $\wellord$, then the square
\[
\xymatrix{
X_{\wo,0}\ar[r]\ar[d]
&X_{\wo,1}\ar[d]
\\
X_{\Succ{\wo},0}\ar[r]
&X_{\Succ{\wo},1}
}
\]
is $\Classarr$-special;
\item[{\rm ii)}] if $\bgwo\neq\minwellord$ is not the successor of an element of $\wellord$, then $X_\bgwo\simeq\varinjlim_{\smwo<\bgwo}X_\smwo$.
\end{itemize}
Then the square
\[
\xymatrix{
X_{\minwellord,0}\ar[r]\ar[d]
&X_{\minwellord,1}\ar[d]
\\
\varinjlim_{\wo\in\wellord}X_{\wo,0}\ar[r]
&\varinjlim_{\wo\in\wellord}X_{\wo,1}
}
\]
is $\Classarr$-special.
\end{lemme}

\begin{proof}
The assertions (\emph{a}) and (\emph{b}) are tautological. Let's prove (\emph{c}). Suppose given a functor $X:\wellord\to\Ar{\CC}$ as in (\emph{c}). By hypothesis~(\emph{i}), if $\Succ{\wo}$ is the successor of some $\wo$ in $\wellord$, there exists a commutative square
\[
\xymatrix{
&A_\wo\ar[r]^{i_\wo}\ar[d]
&B_\wo\ar[d]
\\
&X_{\Succ{\wo},0}\ar[r]
&X_{\Succ{\wo},1}
&\hskip -30pt,
}
\]
with $i_\wo$ in $\imds{\Classarr}$, such that the square
\[
\xymatrix{
A_\wo \amalg X_{\wo,0} \ar[r]^{}\ar[d]
&B_\wo\amalg X_{\wo,1}\ar[d]
\\
X_{\Succ{\wo},0}\ar[r]
&X_{\Succ{\wo},1}
}
\]
is cocartesian. For every $\wo$ in $\wellord$, \emph{choose} such a commutative square. We will prove by transfinite induction on $\wo$ that the ``evident'' commutative square
\[
\raise 20pt
\vbox{
\xymatrix{
\Bigl(\mathop{\amalg}\limits_{\smwo<\bgwo}A_\smwo\Bigr)\amalg X_{\minwellord,0}\ar[r]\ar[d]
&\Bigl(\mathop{\amalg}\limits_{\smwo<\bgwo}B_\smwo\Bigr)\amalg X_{\minwellord,1}\ar[d]
\\
X_{\wo,0}\ar[r]
&X_{\wo,1}
}
}
\leqno(*_\wo)
\]
is cocartesian, which will prove the assertion by stability of cocartesian squares under colimits. For $\wo=\minwellord$, there is nothing to prove. Suppose that the square $(*_\wo)$ is cocartesian for some $\wo$ in $\wellord$, and let $\Succ{\wo}$ be the successor of $\wo$. Consider the commutative diagram
\[
\xymatrix{
&\Bigl(\mathop{\amalg}\limits_{\smwo<\bgwo}A_\smwo\Bigr) \amalg A_\wo \amalg X_{\minwellord,0}\ar[r]\ar[d]
&\Bigl(\mathop{\amalg}\limits_{\smwo<\bgwo}B_\smwo\Bigr) \amalg B_\wo \amalg X_{\minwellord,1}\ar[d]
\\
&A_\wo\amalg X_{\wo,0}\ar[r]\ar[d]
&B_\wo\amalg X_{\wo,1}\ar[d]
\\
&X_{\Succ{\wo},0}\ar[r]
&X_{\Succ{\wo},1}
&\hskip -50pt.
}
\]
The upper square is cocartesian as sum of two cocartesian squares, and the lower one is cocartesian by hypothesis. So the composite square $(*_{\Succ{\wo}})$ is cocartesian. Finally if $\bgwo\neq\minwellord$ is not the successor of an element of $\wellord$, and if for every $\smwo<\bgwo$, $(*_\smwo)$~is cocartesian, then $(*_\bgwo)$ is cocartesian by stability of cocartesian squares under colimits.
\end{proof}

\begin{prop} \label{propclef}
Let $\CC$ be a cocomplete category and $\Setarr$ a (small) set of arrows of $\CC$. If the domains of the arrows in $\Setarr$ are finitely presentable, then \emph{(in the notations of~\ref{remlemmeclef})}
\[
\cell[\omega]{I}=\cell{I}\smsp.
\]
\end{prop}

\begin{proof}
The inclusion
\[
\cell[\omega]{I}\subset\cell{I}\smsp
\]
being trivial, it's enough to prove the opposite inclusion. Recall that $\cell{I}$ is equal to the class of maps obtained by transfinite composition of pushouts of arrows in~$\Setarr$ (\cf~\ref{celmaps}). So, let $\wellord$ be a well-ordered set, with minimal element $\minwellord$, and
\[
(X_\wo)_{\wo\in\wellord}\smsp,\quad(X_{\smwo}\to X_\bgwo)_{\smwo\leqslant\bgwo}
\]
a $\wellord$-indexed inductive system with values in $\CC$ (functor $\wellord\to\CC$), satisfying the following two conditions:
\begin{itemize}
\item[\emph{a})] if $\Succ{\wo}$ is the successor of $\wo$ in $\wellord$, then there exists a cocartesian (pushout) square
\[
\xymatrix{
A_{\wo}\ar[r]^{a_{\wo}}\ar[d]_{i_\wo}
&X_\wo\ar[d]
\\
B_\wo\ar[r]_{b_\wo}
&X_{\Succ{\wo}}
}
\]
with $i_\wo$ in $\Setarr$;
\item[\emph{b})] if $\bgwo\neq\minwellord$ is not the successor of an element of $\wellord$, then $X_\bgwo\simeq\varinjlim_{\smwo<\bgwo}X_\smwo$.
\end{itemize}
We have to prove that the canonical map $X_\minwellord\to\varinjlim_{\wo\in\wellord}X_\wo$ is in $\cell[\omega]{\Setarr}$.
\smallbreak

In the sequel, we suppose that for every $\wo\in\wellord$ we have \emph{chosen} a cocartesian square as in~(\emph{a}).
\smallbreak

We will define by transfinite induction a $\wellord$\nobreakdash-indexed inductive system
\[
(Y_\wo)_{\wo\in\wellord}\smsp,\quad(Y_{\smwo}\to Y_\bgwo)_{\smwo\leqslant\bgwo}\smsp,
\]
where for every $\wo\in\wellord$,
\[
Y_\wo=\ \ Y_{\wo,0}\to Y_{\wo,1}\to Y_{\wo,2}\to\cdots\to Y_{\wo,n}\to Y_{\wo,n+1}\to\cdots
\]
is a sequence of maps in $\CC$, endowed with an isomorphism $\varinjlim_{n\in\mathbb{N}}Y_{\wo,n}\tosim X_\wo$, natural in $\wo$, satisfying the following conditions:
\begin{itemize}
\item[0)] for every $\wo\in\wellord$, $Y_{\wo,0}=X_\minwellord$ and $Y_{\smwo,0}\to Y_{\bgwo,0}$, $\smwo\leqslant\bgwo$, is the identity, and for every $n\in\mathbb{N}$, $Y_{\minwellord,n}\to Y_{\minwellord,n+1}$ is in $\imds{\Setarr}$;
\item[1)] if $\Succ{\wo}$ is the successor of $\wo$ in $\wellord$, then for every $n\in\mathbb{N}$,
\[
\xymatrix{
Y_{\wo,n}\ar[r]\ar[d]
&Y_{\wo,n+1}\ar[d]
\\
Y_{\Succ{\wo},n}\ar[r]
&Y_{\Succ{\wo},n+1}
}
\]
is a $\Setarr$-special square;
\item[2)] if $\bgwo\neq\minwellord$ is not the successor of an element of $\wellord$, then $Y_\bgwo\simeq\varinjlim_{\smwo<\bgwo}Y_\smwo$.
\end{itemize}
This will prove the proposition. Indeed, if we set
\[
Y\!:\,=\varinjlim_{\wo\in\wellord}Y_\wo=\ \Bigl(Y_0\!:\,=\varinjlim_{\wo\in\wellord}Y_{\wo,0}\toto 
                                              Y_1\!:\,=\varinjlim_{\wo\in\wellord}Y_{\wo,1}\toto 
                                              Y_2\!:\,=\varinjlim_{\wo\in\wellord}Y_{\wo,2}\toto\cdots\Bigr)\smsp,
\]
then by lemma~\ref{proprspecial}, (\emph{b}), (\emph{c}), for every $n\in\mathbb{N}$, $Y_n\to Y_{n+1}$ will be in $\imds{\Setarr}$ and therefore the canonical map
\[
X_\minwellord\simeq\varinjlim_{\wo\in\wellord}Y_{\wo,0}=Y_0
{\hskip -2.5pt\xymatrixcolsep{2.pc}\xymatrix{\ar[r]&}\hskip -2.5pt}
\varinjlim_{n\in\mathbb{N}}Y_n= \varinjlim_{n\in\mathbb{N}} \varinjlim_{\wo\in\wellord} Y_{\wo,n}
\simeq\varinjlim_{\wo\in\wellord}\varinjlim_{n\in\mathbb{N}}Y_{\wo,n}\simeq\varinjlim_{\wo\in\wellord}X_\wo 
\]
will be in $\cell[\omega]{\Setarr}$.
\smallbreak

So let's construct such an inductive system. Define $Y_\minwellord$ by
\[
Y_\minwellord\!:\,=\ 
\xymatrixcolsep{1.pc}
\xymatrix{
X_\minwellord\ar[r]^-{=}
&X_\minwellord\ar[r]^-{=}
&X_\minwellord\ar[r]
&\cdots\ar[r]
&X_\minwellord\ar[r]^-{=}
&X_\minwellord\ar[r]
&\cdots\smsp,
}
\]
\ssmash{$\varinjlim_{n\in\mathbb{N}}Y_{\minwellord,n}=\varinjlim_{n\in\mathbb{N}}X_{\minwellord}\tosim X_\minwellord$} being the evident isomorphism. If $\bgwo\neq\minwellord$ is not the successor of an element of $\wellord$, and if $Y_\smwo$ and the isomorphism \ssmash{$\varinjlim_{n\in\mathbb{N}}Y_{\smwo,n}\tosim X_\smwo$} are defined for $\smwo<\bgwo$, define the sequence $Y_\bgwo$ by \ssmash{$Y_\bgwo\!:\,=\varinjlim_{\smwo<\bgwo}Y_\smwo$}, and the isomorphism \ssmash{$\varinjlim_{n\in\mathbb{N}}Y_{\bgwo,n}\tosim X_\bgwo$} as the colimit of the isomorphisms \ssmash{$\varinjlim_{n\in\mathbb{N}}Y_{\smwo,n}\tosim X_\smwo$}. Suppose now that
\[
Y_\wo=\ \ Y_{\wo,0}\to Y_{\wo,1}\to Y_{\wo,2}\to\cdots\to Y_{\wo,n}\to Y_{\wo,n+1}\to\cdots\smsp,\quad
\varinjlim_{n\in\mathbb{N}}Y_{\wo,n}\tosim X_\wo
\]
are defined, and let's define the sequence $Y_{\Succ{\wo}}$, where $\Succ{\wo}$ is the successor of $\wo$, and construct the dotted part of the diagram
\[
\xymatrixrowsep{1.4pc}
\xymatrixcolsep{1.pc}
\xymatrix{
Y_{\wo,0} \ar[r] \ar@{..>}[d]
&Y_{\wo,1}\ar[r]\ar@{..>}[d]
&\cdots\ar[r]
&Y_{\wo,n}\ar[r]\ar@{..>}[d]
&Y_{\wo,n+1}\ar[r]\ar@{..>}[d]
&\cdots\smsp,
&\varinjlim_{n\in\mathbb{N}}Y_{\wo,n}\ar[r]^-{\sim}\ar@{..>}[d]
&X_\wo\ar[d]
\\
Y_{\Succ{\wo},0} \ar@{..>}[r]
&Y_{\Succ{\wo},1}\ar@{..>}[r]
&\cdots\ar@{..>}[r]
&Y_{\Succ{\wo},n}\ar@{..>}[r]
&Y_{\Succ{\wo},n+1}\ar@{..>}[r]
&\cdots\smsp,
&\varinjlim_{n\in\mathbb{N}}Y_{\Succ{\wo},n}\ar@{..>}[r]^-{\sim}
&X_{\Succ{\wo}}
&\hskip -29pt.
}
\]
Recall that we have a cocartesian (pushout) square
\[
\raise 20pt
\vbox{
\xymatrix{
A_{\wo}\ar[r]^{a_{\wo}}\ar[d]_{i_\wo}
&X_\wo\ar[d]
\\
B_\wo\ar[r]_{b_\wo}
&X_{\Succ{\wo}}
}
}
\leqno(*)
\]
with $i_\wo$ in $\Setarr$. As $A_\wo$ is of finite presentation, there exists $n_0\in\mathbb{N}$ such that $a_\wo$ factorizes through $Y_{\wo,n_0}$, and we can choose a minimal such $n_0$.
\[
\xymatrixcolsep{3pc}
\xymatrixrowsep{1.3pc}
\xymatrix{
&Y_{\wo,n_0}\ar[d]^{\mathrm{can}}
\\
A_\wo\ar[rd]_{a_\wo}\ar@{-->}[ru]
&\varinjlim Y_{\wo,n}\ar[d]^{\wr}
\\
&X_\wo
}
\]
For every $n\leqslant n_0$, set $Y_{\Succ{\wo},n}\!:\,=Y_{\wo,n}$, the arrow $Y_{\wo,n}\to Y_{\Succ{\wo},n}$ being the identity. The remaining part of the diagram is defined by constructing the following solid diagram of cocartesian (pushout) squares
\[
\xymatrix{
&A_\wo \ar[r] \ar[d]_{i_\wo}
&Y_{\wo,n_0}\ar[r]\ar[d]\ar@{-->}[rd]
&Y_{\wo,n_0+1}\ar[r]\ar[d]
&Y_{\wo,n_0+2}\ar[r]\ar[d]
&\cdots
\\
&B_\wo\ar[r]
&Z\ar[r]
&Y_{\Succ{\wo},n_0+1}\ar[r]
&Y_{\Succ{\wo},n_0+2}\ar[r]
&\cdots
&\hskip -30pt,
}
\]
the map $Y_{\Succ{\wo},n_0}=Y_{\wo,n_0}\to Y_{\Succ{\wo},n_0+1}$ being the above dotted arrow. The isomorphism $\varinjlim_{n\in\mathbb{N}}Y_{\Succ{\wo},n}\tosim X_{\Succ{\wo}}$ is deduced from the limit cocartesian square
\[
\xymatrixcolsep{1.8pc}
\xymatrix{
&A_\wo\ar[d]_{i_\wo}\ar[rr]
&&\varinjlim_{n>n_0}Y_{\wo,n}\simeq\varinjlim_{n\in\mathbb{N}}Y_{\wo,n}\ar@<-2em>[d]
\\
&B_\wo\ar[rr]
&&\varinjlim_{n>n_0}Y_{\Succ{\wo},n}\simeq\varinjlim_{n\in\mathbb{N}}Y_{\Succ{\wo},n}
&\hskip -30pt,
}
\]
the cocartesian square $(*)$, and the isomorphism $\varinjlim_{n\in\mathbb{N}}Y_{\wo,n}\tosim X_\wo$. It remains to prove that the squares
\[
\xymatrix{
Y_{\wo,n}\ar[r]\ar[d]
&Y_{\wo,n+1}\ar[d]
\\
Y_{\Succ{\wo},n}\ar[r]
&Y_{\Succ{\wo},n+1}
}
\]
are $\Setarr$\nobreakdash-special. For $n\neq n_0$ these squares are cocartesian, hence $\Setarr$\nobreakdash-special by 
lemma~\ref{proprspecial}, (\emph{a}). If $n=n_0$, this is a consequence of lemma~\ref{carspecial} and the above construction.
\end{proof}

The same proof with only minor changes shows the more general result stated in~\ref{remlemmeclef}.

\end{document}